\documentclass[12pt]{article}

\usepackage{amsmath}
\usepackage{amsfonts}
\usepackage{amssymb}

\newtheorem{theorem}{Theorem}

\newtheorem{definition}[theorem]{Definition}

\newtheorem{lemma}[theorem]{Lemma}

\newtheorem{proposition}[theorem]{Proposition}
\newtheorem{remark}[theorem]{Remark}

\newenvironment{proof}[1][Proof]{\textbf{#1.} }{\ \rule{0.5em}{0.5em}}
\makeatletter
\renewcommand\@makefnmark{\relax}
\makeatother

\begin{document}

\title{A local limit theorem for triple connections in subcritical Bernoulli percolation}
\author{M. Campanino\thanks{Partially supported by Italian G.N.A.M.P.A. and the
University of Bologna Funds for selected research topics.}\ , M.
Gianfelice$^{*}$\\Dipartimento di Matematica\\Universit\`{a} degli Studi di Bologna\\P.zza di Porta San Donato, 5 \ I-40127}
\maketitle
\begin{abstract}
We prove a local limit theorem for the probability of a site to be connected
by disjoint paths to three points in subcritical Bernoulli percolation on
$\mathbb{Z}^{d},\,d\geq2$ in the limit where their distances tend to infinity.
\end{abstract}

\footnotetext{
\par
\ \hspace*{0.2cm}\emph{Keywords and phrases}: Percolation, local limit
theorem, decay of connectivities, multidimensional renewal process.}

\section{Introduction and results}

The asymptotic behaviour of the connection function for Bernoulli sub-critical
percolation on $d$-dimensional lattices and of two points correlation
functions of finite range Ising models above critical temperature has been
recently completely proved to agree with that predicted by Ornstein and
Zernike (\cite{CI}, \cite{CIV}, see also \cite{AL}, and \cite{CCC} for some
previous results and \cite{BF} for some results for extreme values of
parameters). The arguments of \cite{CI} and \cite{CIV} follow a general scheme
that is exposed in \cite{CIV1}. A natural question that arises is how higher
order percolation or correlation functions behave for these models. It is
natural to start by addressing this problem in the simplest case, i.e. for
triple connection functions in Bernoulli subcritical percolation on
$d$-dimensional lattices. This analysis is carried out in this work. It turns
out that the techniques developed in \cite{CI} and \cite{CIV}, plus some extra
ideas, allow to obtain the asymptotic behaviour. Interestingly and luckily
enough, some techniques introduced in \cite{CIV} for the Ising models result
useful for our work, though for somewhat different reasons. Besides the
asymptotic behaviour of the probability of triple connections we obtain a
local limit theorem for the positions of points from which three disjoint
paths start and give rise to triple connections. It is worth to observe that
these positions are not decomposed in a natural way as sums of random
variables, as it is most common in local limit theorems.

We consider a Bernoulli bond percolation process on $\mathbb{Z}^{d},\,d\geq2,$
in the subcritical regime ($p<p_{c}\left(  d\right)  $). A basic result,
established by Menshikov \cite{M} and by Aizenman and Barsky \cite{AB} with
different methods, states that, for $p<p_{c}\left(  d\right)  ,$ connection
functions decay exponentially in every direction. Using the FKG inequality it
can be shown that, given a point $n$ of the lattice, the probability
$\mathbb{P}_{p}\left\{  0\leftrightarrow n\right\}  $ that $n$ is connected to
the origin $0$ (i.e. that there exists a chain of open bonds leading from the
origin to $n$) is bounded from above by $e^{-\xi_{p}\left(  n\right)  },$
where
\begin{equation}
\xi_{p}\left(  x\right)  :=-\lim_{N\uparrow\infty}\frac{1}{N}\log
\mathbb{P}_{p}\left\{  0\leftrightarrow\lbrack xN]\right\}  .
\end{equation}
$\xi_{p},$ is always defined and is a finite, convex, homogeneous-of-order-one
function on $\mathbb{R}^{d},$ invariant under permutation and reflection
across coordinate hyperplanes. For $\left|  \left|  x\right|  \right|
=1,\,\xi_{p}$ goes by the name of \emph{inverse connection length} in the
direction $x.$

Let us denote with $\left(  \cdot,\cdot\right)  $ the scalar product in
$\mathbb{R}^{d},$ and with $\left|  \left|  \cdot\right|  \right|
:=\sqrt{\left(  \cdot,\cdot\right)  }$ the associated Euclidean norm. It has
been proved by Hammersley (\cite{G} Theorem 5.1) that if $p<p_{c}\left(
d\right)  $ there exists a strictly positive function $c_{-}\left(  p\right)
$ such that
\begin{equation}
\xi_{p}\left(  x\right)  \geq c_{-}\left(  p\right)  \left|  \left|  x\right|
\right|  \quad x\in\mathbb{R}^{d}\backslash\left\{  0\right\}  ,\label{lbxi}%
\end{equation}
while from Harris inequality it\ follows that
\begin{equation}
\xi_{p}\left(  x\right)  \leq c_{+}\left(  p\right)  \left|  \left|  x\right|
\right|  \quad x\in\mathbb{R}^{d}\backslash\left\{  0\right\}  ,\label{ubxi}%
\end{equation}
which implies that the inverse correlation length is an equivalent norm on
$\mathbb{R}^{d}.$

Following a previous work by \cite{CCC}, where only the axes directions were
considered, recently Campanino and Ioffe showed (\cite{CI} Theorem A) that if
the lattice dimension $d$ is larger than or equal to $2$, uniformly in
$x\in\mathbb{S}^{d-1},$ the correct asymptotics for the \emph{connectivity
function }$\mathbb{P}_{p}\left\{  0\leftrightarrow\left[  Nx\right]  \right\}
$ for $p<p_{c}\left(  d\right)  $ is given by
\begin{equation}
\mathbb{P}_{p}\left\{  0\leftrightarrow\left[  Nx\right]  \right\}
=\frac{\Psi_{p}\left(  x\right)  }{\sqrt{\left(  2\pi N\right)  ^{d-1}}%
}e^{-\xi_{p}\left(  [Nx]\right)  }\left(  1+o\left(  1\right)  \right)
,\label{CIe}%
\end{equation}
where $\Psi_{p}$ is a positive real analytic function on $\mathbb{S}^{d-1}.$

Let
\begin{equation}
\mathbf{U}^{p}:=\left\{  x\in\mathbb{R}^{d}:\xi_{p}\left(  x\right)
\leq1\right\}
\end{equation}
be the unit ball in the $\xi_{p}$-norm ($\xi_{p}$-ball), then any $\xi_{p}
$-ball will be denoted by
\begin{equation}
a\mathbf{U}^{p}:=\left\{  x\in\mathbb{R}^{d}:\xi_{p}\left(  x\right)  \leq
a\right\}  \quad a\in\mathbb{R}^{+}.
\end{equation}
We also introduce the polar body of $\mathbf{U}^{p}$
\begin{equation}
\mathbf{K}^{p}:=\bigcap_{x\in\mathbb{S}^{d-1}}\left\{  t\in\mathbb{R}%
^{d}:\left(  t,x\right)  \leq\xi_{p}\left(  x\right)  \right\}
\end{equation}
Then, given any $x\in\mathbb{R}^{d},$ the set of vectors $t\in\partial
\mathbf{K}^{p}$ meeting the equality
\begin{equation}
\left(  t,x\right)  =\xi_{p}\left(  x\right) \label{defpr}%
\end{equation}
are said to be polar to $x.$ It has been shown (\cite{CI} Lemma 4.3) that both
$\partial\mathbf{U}^{p}$ and $\partial\mathbf{K}^{p}$ are strictly convex
analytic surfaces with gaussian curvature bounded away from zero, so there
exists only one point $t_{x}\in\partial\mathbf{K}^{p}$ satisfying the equality
(\ref{defpr}).

In this paper, using the tools introduced in \cite{CI}, we will analyse the
probability that three distinct points of the lattice are connected through
disjoint open paths, in the limit as their mutual distance tends to infinity.
To this aim we need to introduce some additional notation.

For $x\in\mathbb{R}^{d},$ let us denote by $[x]$ the vector $\left(
[x_{1}],..,[x_{d}]\right)  $ and define
\begin{equation}
X_{3}:=\left\{  \left(  x_{1},x_{2},x_{3}\right)  \in\mathbb{R}^{3d}:x_{i}\neq
x_{j}\,\text{if}\,i\neq j;\,i,j=1,2,3\right\}  .
\end{equation}
Hence, for $\mathbf{x\in}X_{3}$, we define
\begin{equation}
\varphi_{p,\mathbf{x}}\left(  x\right)  :=\sum_{i=1}^{3}\xi_{p}\left(
x-x_{i}\right)  .
\end{equation}
$\varphi_{p,\mathbf{x}}\left(  x\right)  $ is easily seen to be a convex
function whose unique minimum, which is a function of $\mathbf{x},$ we will
denote by $x_{0}\left(  \mathbf{x}\right)  .$ In the following, we will
consider only those elements of $X_{3}$ satisfying the further condition:
\begin{equation}
u_{i}:=\sum_{j\neq i}\nabla\xi_{p}\left(  x_{j}-x_{i}\right)  \notin
\mathbf{K}^{p}\quad\forall i=1,2,3.\label{cond1}%
\end{equation}
The geometrical meaning of (\ref{cond1}) relies on the fact that, given
$\mathbf{x\in}X_{3},$ this condition prevents $x_{0}\left(  \mathbf{x}\right)
$ to coincide with one of the entries of $\mathbf{x}.$ Let then $X_{3}%
^{\prime}$ be the subset of $X_{3}$ whose elements satisfy (\ref{cond1}) and,
given three distinct vertices $n_{1},n_{2},n_{3}$ of the lattice, let:

\begin{itemize}
\item
\begin{equation}
E\left(  n_{1},n_{2},n_{3}\right)  =E\left(  \mathbf{n}\right)
\end{equation}
be the event that $n_{1},n_{2},n_{3}$ are connected by an open cluster;

\item
\begin{equation}
F\left(  k;n_{1},n_{2},n_{3}\right)  =F\left(  k;\mathbf{n}\right)  \quad
k\in\mathbb{Z}^{d}%
\end{equation}
be the event that $k$ is connected by three disjoint self-avoiding open paths
$\gamma_{1},\gamma_{2},\gamma_{3}$ to $n_{1},n_{2},n_{3}$ respectively.
\end{itemize}

Then we have

\begin{theorem}
\label{main}Let $\mathbf{x}\in X_{3}^{\prime},\,y\in\mathbb{R}^{d}$ and let
$N$ vary over the integers. If we denote by $x_{0}\left(  \mathbf{x}\right)  $
the minimizing point of the function $\varphi_{p,\mathbf{x}},$ then, for
$d\geq2$ and $p<p_{c}\left(  d\right)  ,$
\begin{gather}
\mathbb{P}_{p}\left[  F\left(  \left[  x_{0}\left(  \mathbf{x}\right)
N+y\sqrt{N}\right]  ;\left[  N\mathbf{x}\right]  \right)  |E\left(  \left[
N\mathbf{x}\right]  \right)  \right]  =\label{PFE}\\
\Phi_{p}\left(  \mathbf{x}\right)  \frac{\sqrt{\det H_{\varphi}\left(
x_{0}\left(  \mathbf{x}\right)  ,\mathbf{x};p\right)  }}{\left(  2\pi
N\right)  ^{\frac{d}{2}}}\exp\left[  -\frac{\left(  y,H_{\varphi}\left(
x_{0}\left(  \mathbf{x}\right)  ,\mathbf{x};p\right)  y\right)  }{2}\right]
\left(  1+o\left(  1\right)  \right)  ,\nonumber
\end{gather}
where $H_{\varphi}\left(  x_{0}\left(  \mathbf{x}\right)  ,\mathbf{x}%
;p\right)  $ is the Hessian matrix of the function $\varphi_{p,\mathbf{x}}$
evaluated at $x_{0}\left(  \mathbf{x}\right)  ,$ and $\Phi_{p}\left(
\mathbf{x}\right)  $ is an analytic function on $X_{3}^{\prime}.$
\end{theorem}

\begin{remark}
\label{rem1}For any $\varepsilon\in\left(  0,\frac{1}{2}\right)  $ and any
$\beta\in\left(  0,\frac{1}{2}\right)  ,$ let
\begin{multline}
F_{\varepsilon,\beta}\left(  k_{1},k_{2};[N\mathbf{x}]\right)  :=F\left(
k_{1};[Nx]\right)  \cap F\left(  k_{2};[Nx]\right)  \cap\left\{  k_{1}%
,k_{2}\in N^{\frac{1}{2}+\varepsilon}\mathbf{U}^{p}\left(  Nx_{0}\left(
\mathbf{x}\right)  \right)  \cap\mathbb{Z}^{d}:\right. \nonumber\\
\left.  \left|  \left|  k_{1}-k_{2}\right|  \right|  >N^{\beta}\right\}
\end{multline}
be the event that two lattice's points $k_{1}$ and $k_{2},$ belonging to a
$\xi_{p}$-neighborhood of $[Nx_{0}\left(  \mathbf{x}\right)  ]$ of radius
$N^{\frac{1}{2}+\varepsilon}$ and whose mutual distance is larger than
$N^{\beta},\,$are connected to $[Nx_{1}],[Nx_{2}],[Nx_{3}]$ by three disjoint
self-avoiding open paths. As a byproduct, in the proof of Theorem \ref{main}
we also get that there exists a positive constant $c^{\prime\prime}$ such
that,
\begin{equation}
\mathbb{P}_{p}\left[  F_{\varepsilon,\beta}\left(  k_{1},k_{2};[N\mathbf{x}%
]\right)  \right]  \leq e^{-N\varphi_{p,\mathbf{x}}\left(  x_{0}\left(
\mathbf{x}\right)  \right)  -c^{\prime\prime}N^{\beta\wedge2\varepsilon}}.
\end{equation}
\end{remark}

\section{Local limit theorem}

\subsection{Preliminary results}

Let us define
\begin{equation}
\mathcal{H}_{y}^{t}:=\left\{  x\in\mathbb{R}^{d}:\left(  t,x\right)  =\left(
t,y\right)  \right\}  \quad y\in\mathbb{R}^{d}%
\end{equation}
to be the $\left(  d-1\right)  $-dimensional hyperplane in $\mathbb{R}^{d}$
orthogonal to the vector $t$ passing through a point $y\in\mathbb{R}^{d}$ and
the corresponding halfspaces
\begin{align}
\mathcal{H}_{y}^{t,-}  &  :=\left\{  x\in\mathbb{R}^{d}:\left(  t,x\right)
\leq\left(  t,y\right)  \right\}  ,\\
\mathcal{H}_{y}^{t,+}  &  :=\left\{  x\in\mathbb{R}^{d}:\left(  t,x\right)
\geq\left(  t,y\right)  \right\}  .
\end{align}
Then we have

\begin{lemma}
\label{lbqf}For any $\mathbf{x\in}X_{3}^{\prime},\,\varphi_{p,\mathbf{x}}$ is
a strictly convex function on a neighborhood of $x_{0}\left(  \mathbf{x}%
\right)  ,$ where it is lower bounded by a strictly positive quadratic form of
$x-x_{0}\left(  \mathbf{x}\right)  .$
\end{lemma}

\begin{proof}
Setting $y=x-x_{0}\left(  \mathbf{x}\right)  $ we consider $\varphi
_{p,\mathbf{y}}\left(  y\right)  =\sum_{i=1}^{3}\xi_{p}\left(  y-y_{i}\right)
,$ where\linebreak \ $y_{i}=x_{i}-x_{0}\left(  \mathbf{x}\right)  .$ Let
$t_{i}\in\partial\mathbf{K}^{p}$ be the polar point to $y_{i}.$ By the
convexity of $\xi_{p}$ and \cite{CI} Lemma 4.4, there exist positive constants
$c^{\prime},c$ such that, for any $z\in\mathbb{R}^{d}$ satisfying, $\left(
z,t_{i}\right)  =\left(  y_{i},t_{i}\right)  =\xi_{p}\left(  y_{i}\right)  $
and $\left|  \left|  z-y_{i}\right|  \right|  \leq c^{\prime},$
\begin{equation}
\xi_{p}\left(  z\right)  \geq\left(  t_{i},z\right)  +c\left|  \left|
z-y_{i}\right|  \right|  ^{2}.
\end{equation}
Hence, for any $y\in\mathbb{R}^{d}$ such that $\left|  \left|  y\right|
\right|  \leq c^{\prime},$ setting $z=y_{i}-y,$ we get
\begin{equation}
\xi_{p}\left(  y-y_{i}\right)  -\xi_{p}\left(  y_{i}\right)  \geq-\left(
\nabla\xi_{p}\left(  y_{i}\right)  ,y\right)  +c\left|  \left|  P_{i}^{\bot
}y\right|  \right|  ^{2}\quad i=1,2,3,
\end{equation}
where, $\forall i=1,2,3,\,P_{i}^{\bot}$ is the orthogonal projector on
$\mathcal{H}_{0}^{t_{i}}$. Summing up, since by the definition of
$x_{0}\left(  \mathbf{x}\right)  ,$\thinspace%
\begin{equation}
\sum_{i=1}^{3}\nabla\xi_{p}\left(  y_{i}\right)  =\sum_{i=1}^{3}\nabla\xi
_{p}\left(  x_{0}\left(  \mathbf{x}\right)  -x_{i}\right)  =0,\label{condmin}%
\end{equation}
we get
\begin{equation}
\varphi_{p,\mathbf{y}}\left(  y\right)  -\varphi_{p,\mathbf{y}}\left(
0\right)  \geq c\sum_{i=1}^{3}\left|  \left|  P_{i}^{\bot}y\right|  \right|
^{2}.\label{strconv}%
\end{equation}
The right hand side of the last expression can never be zero for $y\neq0$
because, $\forall i=1,2,3,$ the hyperplanes $\mathcal{H}_{0}^{t_{i}}$ have
codimension one and conditions (\ref{cond1}) and (\ref{condmin}) prevent the
vectors $\nabla\xi_{p}\left(  x_{0}\left(  \mathbf{x}\right)  -x_{i}\right)
,\,i=1,2,3,$ from being parallel.
\end{proof}

For$\,l\geq1,$ let $\mathbf{C}_{\{k_{1},..,k_{l}\}}$ denote the common open
cluster of the points $k_{1},..,k_{l}\in\mathbb{Z}^{d},$ provided it exists,
and let $t\in\partial\mathbf{K}^{p}.$ Given two points $k_{i},k_{j} $ such
that $\left(  k_{i},t\right)  \leq\left(  k_{j},t\right)  $, we denote by
$\mathbf{C}_{\{k_{i},k_{j}\}}^{t}$ the cluster of $k_{i}$ and $k_{j}$ inside
the strip $\allowbreak\mathcal{S}_{\left\{  k_{i},k_{j}\right\}  }%
^{t}:=\mathcal{H}_{k_{i}}^{t,+}\cap\mathcal{H}_{k_{j}}^{t,-}.$

First we estimate the probability that, $\forall i=1,2,3,$ the points
$[Nx_{i}]$ are connected through three disjoint open paths to a point whose
distance from $x_{0}\left(  [N\mathbf{x}]\right)  $ is larger than $N^{\alpha}
$ with $\alpha\in\left(  \frac{1}{2},1\right)  .$

For any $\mathbf{x}\in X_{3}^{\prime},$ let $\mathbf{C}_{[N\mathbf{x}%
]}=\mathbf{C}_{\left\{  [Nx_{1}],[Nx_{2}],[Nx_{3}]\right\}  }$ and
\begin{equation}
A_{\alpha,N}\left(  \mathbf{x}\right)  :=\left\{  \exists n\in\mathbf{C}%
_{[N\mathbf{x}]}:n\overset{\gamma_{i}}{\longleftrightarrow}[Nx_{i}%
],\,\gamma_{i}\cap\gamma_{j}=n,\,i,j=1,2,3,\,i\neq j;\,\left|  \left|
n-x_{0}\left(  [N\mathbf{x}]\right)  \right|  \right|  \geq N^{\alpha}\right\}
\label{defA}%
\end{equation}
be the event that the lattice points $[Nx_{i}]$ are connected through three
disjoint open paths to a point $n$ whose distance from $x_{0}\left(
[N\mathbf{x}]\right)  $ is larger than or equal to $N^{\alpha}.$ We have

\begin{proposition}
\label{p1}For any $\mathbf{x}\in X_{3}^{\prime}$ and $\alpha>\frac{1}{2},$
\begin{equation}
\mathbb{P}_{p}[A_{\alpha,N}\left(  \mathbf{x}\right)  ]\leq e^{-\varphi
_{p,[N\mathbf{x]}}\left(  x_{0}\left(  [N\mathbf{x}]\right)  \right)
}e^{-c_{1}N^{2\alpha-1}},\label{3BK}%
\end{equation}
with $c_{1}$ a positive constant.
\end{proposition}

\begin{proof}
By the BK inequality (see e.g. \cite{G})
\begin{equation}
\mathbb{P}_{p}[A_{\alpha,N}\left(  \mathbf{x}\right)  ]\leq e^{-\varphi
_{p,\left[  N\mathbf{x}\right]  }\left(  x_{0}\left(  \left[  N\mathbf{x}%
\right]  \right)  \right)  }\sum_{n\in\mathbb{Z}^{d}\,:\,\left|  \left|
n-x_{0}\left(  [N\mathbf{x}]\right)  \right|  \right|  \geq N^{\alpha}%
}e^{-[\varphi_{p,\left[  N\mathbf{x}\right]  }\left(  n\right)  -\varphi
_{p,[N\mathbf{x}]}\left(  x_{0}\left(  \left[  N\mathbf{x}\right]  \right)
\right)  ]}.\label{BK1}%
\end{equation}
The convexity of $\varphi_{p,\left[  N\mathbf{x}\right]  }$ implies that given
$z\in\mathbb{R}^{d},$ for any point $y$ lying on the segment between $z$ and
$x_{0}\left(  \left[  N\mathbf{x}\right]  \right)  ,$ we have
\begin{equation}
\varphi_{p,\left[  N\mathbf{x}\right]  }\left(  z\right)  -\varphi_{p,\left[
N\mathbf{x}\right]  }\left(  x_{0}\left(  \left[  N\mathbf{x}\right]  \right)
\right)  \geq\frac{\left|  \left|  z-x_{0}\left(  \left[  N\mathbf{x}\right]
\right)  \right|  \right|  }{\left|  \left|  y-x_{0}\left(  \left[
N\mathbf{x}\right]  \right)  \right|  \right|  }\left(  \varphi_{p,\left[
N\mathbf{x}\right]  }\left(  y\right)  -\varphi_{p,\left[  N\mathbf{x}\right]
}\left(  x_{0}\left(  \left[  N\mathbf{x}\right]  \right)  \right)  \right)
.\label{conv1}%
\end{equation}

Since $\xi_{p}$ is a homogeneous function of order one, its Hessian matrix
$H_{\xi}\left(  \cdot;p\right)  $ is a homogeneous function of order $-1.$
Hence, choosing $y$ such that $\left|  \left|  y-x_{0}\left(  \left[
N\mathbf{x}\right]  \right)  \right|  \right|  =N^{\alpha},\,\forall
i=1,2,3,\,\left|  \left|  y-[Nx_{i}]\right|  \right|  \geq N^{\alpha}$ and by
(\ref{strconv}) there exists a positive constant $c_{2}$ such that
\begin{equation}
\varphi_{p,\left[  N\mathbf{x}\right]  }\left(  y\right)  -\varphi_{p,\left[
N\mathbf{x}\right]  }\left(  x_{0}\left(  \left[  N\mathbf{x}\right]  \right)
\right)  \geq c_{2}N^{2\alpha-1}.\label{conv2}%
\end{equation}
Furthermore, for any $z$ outside of a neigbourhood of $[Nx_{i}],\,i=1,2,3,$%
\begin{equation}
\left(  H_{\varphi}\left(  z,\frac{[N\mathbf{x}]}{N};p\right)  \right)
_{i,j}=\left(  H_{\varphi}\left(  z,\mathbf{x};p\right)  \right)
_{i,j}+O\left(  \frac{1}{N}\right)  ,\quad i,j=1,..,d\,.\label{stimHfi}%
\end{equation}
Substituting (\ref{conv1}) and (\ref{conv2}) into (\ref{BK1}), for values of
$N$ large enough, we can bound the r.h.s. of (\ref{BK1}) by
\begin{gather}
\sum_{n\in\mathbb{Z}^{d}\,:\,\left|  \left|  n-x_{0}\left(  \left[
N\mathbf{x}\right]  \right)  \right|  \right|  \geq N^{\alpha}}e^{-c_{2}%
N^{\alpha-1}\left|  \left|  n-x_{0}\left(  \left[  N\mathbf{x}\right]
\right)  \right|  \right|  }\leq c_{3}\int_{\left\{  x\in\mathbb{R}%
^{d}\,:\,\left|  \left|  x\right|  \right|  \geq N^{\alpha}\right\}
}dxe^{-c_{2}N^{\alpha-1}\left|  \left|  x\right|  \right|  }\\
\leq\frac{c_{4}}{N^{\left(  \alpha-1\right)  d}}\int_{c_{2}N^{2\alpha-1}%
}^{\infty}r^{\left(  d-1\right)  }e^{-r}dr\leq c_{4}N^{d\left(  1-\alpha
\right)  }e^{-\frac{c_{2}}{2}N^{2\alpha-1}}\,.\nonumber
\end{gather}
\end{proof}

\subsection{Renewal structure of connectivities}

Given $t\in\partial\mathbf{K}^{p}$ and a positive number $\eta<1$, we define
the set (\emph{surcharge} cone)
\begin{equation}
\mathcal{C}_{\eta}\left(  t\right)  :=\left\{  x\in\mathbb{R}^{d}:\left(
t,x\right)  \geq\left(  1-\eta\right)  \xi_{p}\left(  x\right)  \right\}  .
\end{equation}

We now follow \cite{CI} and \cite{CIV}. Let $e$ be the first of the unit
vectors $e_{1},..,e_{d}$ in the direction of the coordinate axis such that
$\left(  t,e\right)  $ is maximal and let $x_{t}$ denote the element of
$\partial\mathbf{U}^{p}$ polar to $t.$

\begin{definition}
$k,n\in\mathbb{Z}^{d}$ are called $h_{t}$-connected if

\begin{itemize}
\item [1 -]$n$ and $k$ are connected in $\mathcal{S}_{\{k,n\}}^{t};$

\item[2 -]
\begin{equation}
\mathbf{C}_{\{k,n\}}^{t}\cap\mathcal{S}_{\left\{  k,k+e\right\}  }%
^{t}=\left\{  k,k+e\right\}  ,\quad\mathbf{C}_{\{k,n\}}^{t}\cap\mathcal{S}%
_{\left\{  n-e,n\right\}  }^{t}=\left\{  n-e,n\right\}  .
\end{equation}
\end{itemize}
\end{definition}

Moreover, denoting by $\left\{  k\overset{h_{t}}{\longleftrightarrow
n}\right\}  $ the event that $n$ and $k$ are $h_{t}$-connected, we set
\begin{equation}
h_{t}\left(  k,n\right)  :=\mathbb{P}_{p}\left\{  k\overset{h_{t}%
}{\longleftrightarrow n}\right\}  .
\end{equation}

Notice that, by translation invariance, $h_{t}\left(  k,n\right)
=h_{t}\left(  n-k,0\right)  $ so in the sequel we will denote it simply by
$h_{t}\left(  n-k\right)  .$ We also define by convention $h_{t}\left(
0\right)  =1.$

\begin{definition}
Let $k,n\in\mathbb{Z}^{d}$ be connected. The points $b\in\mathbf{C}%
_{\{k,n\}}^{t}$ such that:

\begin{itemize}
\item [1 -]$\left(  t,k+e\right)  \leq\left(  t,b\right)  \leq\left(
t,n-e\right)  ;$

\item[2 -] $\mathbf{C}_{\{k,n\}}^{t}\cap\mathcal{S}_{\left\{  b-e,b+e\right\}
}^{t}=\left\{  b-e,b,b+e\right\}  ;$
\end{itemize}

are said to be $t$\emph{-break points} of $\mathbf{C}_{\{k,n\}}.$ The
collection of such points, which we remark is a totally ordered set with
respect to the scalar product with $t,$ will be denoted by $\mathbf{B}%
^{t}\left(  k,n\right)  .$
\end{definition}

\begin{definition}
Let $k,n\in\mathbb{Z}^{d}$ be $h_{t}$-connected. If $\mathbf{B}^{t}\left(
k,n\right)  $ is empty, then $n$ and $k$ are said to be $f_{t}$-connected and
the corresponding event is denoted by $\left\{  k\overset{f_{t}}%
{\longleftrightarrow n}\right\}  .$ We then set\linebreak \ $f_{t}\left(
n-k\right)  :=\mathbb{P}_{p}\left\{  k\overset{f_{t}}{\longleftrightarrow
n}\right\}  .$
\end{definition}

We define by convention $f_{t}\left(  0\right)  =0.$ We now introduce a
particular subset of $t$-break points (inspired by section 2.6 of \cite{CIV})
which will provide a decomposition of the event $E\left(  [N\mathbf{x}%
]\right)  $ into suitable\ disjoint events.

Let $K$ be a positive constant that will be chosen sufficiently large.

\begin{definition}
$k,n\in\mathbb{Z}^{d}$ are said to be $h_{t}^{\eta,K}$-connected and the
corresponding event is denoted by $\left\{  k\overset{h_{t}^{\eta,K}%
}{\longleftrightarrow n}\right\}  ,$ if $n$ and $k$ are $h_{t}$-connected and
satisfy the following conditions:

\begin{itemize}
\item [1 -]$n\in k+\mathcal{C}_{\eta}\left(  t\right)  ;$

\item[2 -] $\mathbf{C}_{\{k,n\}}^{t}\subseteq\left(  k-Kx_{t}\right)
+\mathcal{C}_{\eta}\left(  t\right)  .$

As for the other connection functions we put $h_{t}^{\eta,K}\left(  0\right)
=h_{t}\left(  0\right)  =1.$
\end{itemize}
\end{definition}

\begin{definition}
\label{defeKbrpt}Let $k,n\in\mathbb{Z}^{d}$ be connected. We define $b_{1}$ to
be the element of the set
\begin{equation}
\left\{  l\in\mathbf{C}_{\left\{  k,n\right\}  }\cap\mathcal{H}_{k}%
^{t,+}:\mathbf{C}_{\left\{  l,n\right\}  }\cap\mathcal{H}_{l}^{t,+}%
\subseteq\left(  l-Kx_{t}\right)  +\mathcal{C}_{\eta}\left(  t\right)
\right\}
\end{equation}
satisfying:

\begin{itemize}
\item [a -]$\mathbf{C}_{\left\{  k,n\right\}  }\cap\mathcal{S}_{\{b_{1}%
-e,b_{1}\}}^{t}=\{b_{1}-e,b_{1}\};$

\item[b -] $\left(  b_{1}-k,t\right)  $ is maximal.

Given $b_{j}$ ($j\geq1$), we denote by $b_{j+1}$ the first$\ t$-break point of
$\mathbf{C}_{\left\{  k,n\right\}  }^{t}$ following $b_{j}$ satisfying the
following conditions:

\item[1 -] $b_{j}\in b_{j+1}+\mathcal{C}_{\eta}\left(  t\right)  ;$

\item[2 -] $\xi_{p}\left(  b_{j}-b_{j+1}\right)  \geq\frac{2K}{\eta};$

\item[3 -] $\mathbf{C}_{\{b_{j+1},b_{j}\}}^{t}\subseteq\left(  b_{j+1}%
-Kx_{t}\right)  +\mathcal{C}_{\eta}\left(  t\right)  ;$

provided it exists. We will call these points $\left(  \eta,K,t\right)
$-break points and denote their collection by $\mathbf{B}^{t}\left(
k,n;\eta,K\right)  .$
\end{itemize}
\end{definition}

\begin{definition}
Any two distinct points $k,n\in\mathbb{Z}^{d}$ are said to be $f_{t}^{\eta,K}%
$-connected if they are $h_{t}^{\eta,K}$-connected and $\mathbf{B}^{t}\left(
k,n;\eta,K\right)  =\emptyset.$
\end{definition}

Clearly, $f_{t}^{\eta,K}\left(  0\right)  =f_{t}\left(  0\right)  =0.$

\begin{lemma}
\label{noint}Let $t\in\partial\mathbf{K}^{p}$ and let $k,n\in\mathbb{Z}^{d}$,
with $\left(  t,n-k\right)  >0,$ be connected. It is possible to choose
$\eta\in\left(  0,1\right)  $ small enough and $K$ sufficiently large such
that, if
\begin{equation}
\mu:=\max\{j\geq2:b_{j}\in\mathbf{B}^{t}\left(  k,n;\eta,K\right)
\},\label{defi}%
\end{equation}
then $\mathbf{C}_{\{b_{\mu-1},n\}}\cap\mathcal{H}_{b_{\mu-1}}^{t,+}\subset
b_{\mu}+\mathcal{C}_{2\eta}\left(  t\right)  .$
\end{lemma}

\begin{proof}
$\forall m\in\mathbf{C}_{\{b_{\mu-1},n\}}\cap\mathcal{H}_{b_{\mu-1}}^{t,+}$ we
set
\begin{equation}
l=l\left(  m\right)  :=\min\left\{  1\leq j\leq\mu-1:\left(  b_{j}-k,t\right)
\leq\left(  m-k,t\right)  \right\}
\end{equation}
and, since $1\leq l\leq\mu-1,$ we consider the following cases:

\begin{enumerate}
\item  If $l=\mu-1,$ by Definition \ref{defeKbrpt}, $m\in b_{\mu-1}%
-Kx_{t}+\mathcal{C}_{\eta}\left(  t\right)  $ and $b_{\mu-1}\in b_{\mu
}+\mathcal{C}_{\eta}\left(  t\right)  .$ Hence,
\begin{gather}
\left(  m-b_{\mu}+Kx_{t},t\right)  =\left(  m-b_{\mu-1}+Kx_{t}+b_{\mu
-1}-b_{\mu},t\right)  =\\
\left(  m-b_{\mu-1}+Kx_{t},t\right)  +\left(  b_{\mu-1}-b_{\mu},t\right)
\geq\nonumber\\
\left(  1-\eta\right)  \xi_{p}\left(  m-b_{\mu-1}+Kx_{t},t\right)  +\left(
1-\eta\right)  \xi_{p}\left(  b_{\mu-1}-b_{\mu}\right)  \geq\nonumber\\
\left(  1-\eta\right)  \xi_{p}\left(  m-b_{\mu}+Kx_{t},t\right) \nonumber
\end{gather}
that is $m\in b_{\mu}-Kx_{t}+\mathcal{C}_{\eta}\left(  t\right)  .$ If $m\in
b_{\mu}+\mathcal{C}_{\eta}\left(  t\right)  ,$ then the thesis is verified.
Otherwise, $\left(  m-b_{\mu},t\right)  \leq\left(  1-\eta\right)  \xi
_{p}\left(  m-b_{\mu}\right)  .$ Therefore,
\begin{align}
\xi_{p}\left(  m-b_{\mu}\right)   &  \geq\frac{\left(  m-b_{\mu},t\right)
}{1-\eta}\geq\frac{\left(  m-b_{\mu-1},t\right)  }{1-\eta}+\xi_{p}\left(
b_{\mu-1}-b_{\mu}\right) \\
&  \geq\xi_{p}\left(  b_{\mu-1}-b_{\mu}\right)  .\nonumber
\end{align}
Moreover,
\begin{gather}
\left(  m-b_{\mu},t\right)  =\left(  m-b_{\mu}+x_{t}K-x_{t}K,t\right)
=\left(  m-b_{\mu}+x_{t}K,t\right)  -K\geq\\
\left(  1-\eta\right)  \xi_{p}\left(  m-b_{\mu}+x_{t}K\right)  -K\geq
\nonumber\\
\left(  1-\eta\right)  \xi_{p}\left(  m-b_{\mu}\right)  -\left(
1-\eta\right)  K-K=\nonumber\\
\left(  1-2\eta\right)  \xi_{p}\left(  m-b_{\mu}\right)  +\eta\xi_{p}\left(
m-b_{\mu}\right)  -2K+\eta K\geq\nonumber\\
\left(  1-2\eta\right)  \xi_{p}\left(  m-b_{\mu}\right)  +\eta\xi_{p}\left(
b_{\mu}-b_{\mu-1}\right)  -2K\,,
\end{gather}
but, by condition 2 of Definition \ref{defeKbrpt}, $\xi_{p}\left(  b_{\mu
-1}-b_{\mu}\right)  \geq\frac{2K}{\eta}.$ Thus $\left(  m-b_{\mu},t\right)
\geq\left(  1-2\eta\right)  \xi_{p}\left(  m-b_{\mu}\right)  .$

\item  If $1\leq l\leq\mu-2,$%
\begin{equation}
m-b_{\mu}=b_{l+1}-b_{\mu}+m-b_{l+1}=m-b_{l+1}+\sum_{j=1}^{\mu-l-1}\left(
b_{l+j}-b_{l+j+1}\right)  ,\label{m-bi-1}%
\end{equation}
then, since by the previous case $m\in b_{l+1}+\mathcal{C}_{2\eta}\left(
t\right)  $ and by condition 1 of Definition \ref{defeKbrpt} $b_{l+j}%
-b_{l+j+1}\in\mathcal{C}_{\eta}\left(  t\right)  \subset\mathcal{C}_{2\eta
}\left(  t\right)  ,$ the r.h.s. of (\ref{m-bi-1}), as a sum of elements of
$\mathcal{C}_{2\eta}\left(  t\right)  $, also belongs to $\mathcal{C}_{2\eta
}\left(  t\right)  .$
\end{enumerate}
\end{proof}

For any $b\in\mathbb{Z}^{d},$ it is easy to see that
\begin{equation}
f_{t}^{\eta,K}\left(  b\right)  \leq h_{t}^{\eta,K}\left(  b\right)  \leq
h_{t}\left(  b\right)  \leq e^{-\xi_{p}\left(  b\right)  }.\label{f<exi}%
\end{equation}

From the previous definitions it follows that $h_{t}^{\eta,K}\left(  n\right)
$ satisfies a renewal equation analogous to the one satisfied by $h_{t}%
$-connected points given in \cite{CI} (4.3), i.e.
\begin{equation}
h_{t}^{\eta,K}\left(  n\right)  =\sum_{b\in\mathbb{Z}^{d}}h_{t}^{\eta
,K}\left(  b\right)  f_{t}^{\eta,K}\left(  n-b\right)  .\label{reneq}%
\end{equation}
Furthermore, it can be shown that$\,h_{t}^{\eta,K}\left(  [Nx]\right)  ,\,$for
$x$ in a neighbourhood of the dual point of $t,$ satisfies the same asymptotic
behaviour of $h_{t}\left(  [Nx]\right)  $ (see \cite{CI} Lemma 4.5). That is,
for any $\eta\in\left(  0,1\right)  $ and $K$ large enough,
\begin{equation}
h_{t}^{\eta,K}\left(  [Nx]\right)  =\frac{\Lambda_{p}\left(  \frac{x}{\left|
\left|  x\right|  \right|  },t\right)  }{\sqrt{2\pi N^{d-1}\left|  \left|
x\right|  \right|  ^{d-1}}}e^{-\xi_{p}\left(  [Nx]\right)  }\left(  1+o\left(
1\right)  \right)  ,\label{asht}%
\end{equation}
where $\Lambda_{p}\left(  \cdot,t\right)  $ is an analytic function on
$\mathbb{S}^{d-1}$ (different from that relative to $h_{t}$ appearing in
\cite{CI} (4.18)). The proof of this assertion relies on arguments similar to
the ones used in \cite{CI}\ to prove the Ornstein-Zernike theory for the
connectivity function and so it will be omitted.

\begin{definition}
Let $k,n\in\mathbb{Z}^{d}$ be connected. Then:

\begin{itemize}
\item [1 -]$k,n$ are called $\bar{h}_{t}$-connected and the corresponding
event is denoted by $\left\{  k\overset{\bar{h}_{t}}{\longleftrightarrow
n}\right\}  ,$ if $\mathbf{C}_{\{k,n\}}\cap\mathcal{S}_{\left\{
n-e,n\right\}  }^{t}=\left\{  n-e,n\right\}  .$

\item[2 -] $k,n$ are called $\bar{f}_{t}^{\eta,K}$-connected and the
corresponding event is denoted by $\left\{  k\overset{\bar{f}_{t}^{\eta,K}%
}{\longleftrightarrow n}\right\}  ,$ if they are $\bar{h}_{t}$-connected and
$\mathbf{B}^{t}\left(  k,n;\eta,K\right)  =\emptyset.$
\end{itemize}
\end{definition}

\begin{definition}
Let $k,n\in\mathbb{Z}^{d}$ be connected. Then:

\begin{itemize}
\item [1 -]$k,n$ are called $\tilde{h}_{t}^{\eta,K}$-connected and the
corresponding event is denoted by $\left\{  k\overset{\tilde{h}_{t}^{\eta,K}%
}{\longleftrightarrow n}\right\}  ,$ if:

\begin{itemize}
\item [1.a -]$\mathbf{C}_{\{k,n\}}\cap\mathcal{H}_{k}^{t,+}\subseteq\left(
k-Kx_{t}\right)  +\mathcal{C}_{\eta}\left(  t\right)  ;$

\item[1.b -] $\mathbf{C}_{\left\{  k,n\right\}  }\cap\mathcal{S}%
_{\{k-e,k\}}^{t}=\{k-e,k\}.$
\end{itemize}

\item[2 -] $k,n$ are called $\tilde{f}_{t}^{\eta,K}$-connected and the
corresponding event is denoted by $\left\{  k\overset{\tilde{f}_{t}^{\eta,K}%
}{\longleftrightarrow n}\right\}  ,$ if they are $\tilde{h}_{t}^{\eta,K}%
$-connected and $\mathbf{B}^{t}\left(  k,n;\eta,K\right)  =\emptyset.$
\end{itemize}
\end{definition}

The probabilities $\mathbb{P}_{p}\left\{  k\overset{\bar{h}_{t}}%
{\longleftrightarrow n}\right\}  :=\bar{h}_{t}\left(  k,n\right)  $ and
$\mathbb{P}_{p}\left\{  k\overset{\tilde{h}_{t}^{\eta,K}}{\longleftrightarrow
n}\right\}  :=\tilde{h}_{t}^{\eta,K}\left(  k,n\right)  $ are translation
invariant, bounded from above by $e^{-\xi_{p}\left(  n-k\right)  }$ and show
an asymptotic behaviour similar to that of $h_{t}^{\eta,K}$ given in
(\ref{asht}), that is there exist two analytic functions on $\mathbb{S}%
^{d-1},\,\bar{\Lambda}_{p}\left(  \cdot,t\right)  $ and $\tilde{\Lambda}%
_{p}\left(  \cdot,t\right)  $\ such that, for $x$ in a neighbourhood of
$x_{t},$
\begin{align}
\bar{h}_{t}\left(  [Nx]\right)   &  =\frac{\bar{\Lambda}_{p}\left(  \frac
{x}{\left|  \left|  x\right|  \right|  },t\right)  }{\sqrt{2\pi N^{d-1}\left|
\left|  x\right|  \right|  ^{d-1}}}e^{-\xi_{p}\left(  [Nx]\right)  }\left(
1+o\left(  1\right)  \right)  ,\\
\tilde{h}_{t}^{\eta,K}\left(  [Nx]\right)   &  =\frac{\tilde{\Lambda}%
_{p}\left(  \frac{x}{\left|  \left|  x\right|  \right|  },t\right)  }%
{\sqrt{2\pi N^{d-1}\left|  \left|  x\right|  \right|  ^{d-1}}}e^{-\xi
_{p}\left(  [Nx]\right)  }\left(  1+o\left(  1\right)  \right)  .\label{ashtt}%
\end{align}
Denoting by $d_{t}^{\eta,K}\left(  k,n\right)  $ the probability of the event
$\{k\longleftrightarrow n,\,\mathbf{B}^{t}\left(  k,n;\eta,K\right)
=\emptyset\},$ which is also translation invariant, we obtain
\begin{equation}
\mathbb{P}_{p}\{0\longleftrightarrow n\}=d_{t}^{\eta,K}\left(  n\right)
+\sum_{b_{1},b_{2}\in\mathbb{Z}^{d}}\bar{f}_{t}^{\eta,K}\left(  b_{2}\right)
h_{t}^{\eta,K}\left(  b_{1}-b_{2}\right)  \tilde{f}_{t}^{\eta,K}\left(
n-b_{1}\right)  .
\end{equation}

\subsection{Renormalization}

We now follow \cite{CI} subsection 2.2 and \cite{CIV} section 2.

Let us represent a self-avoiding open path $\gamma$ connecting the points
$k,n\in\mathbb{Z}^{d}$ by the sequence of points $\left(  n,i_{1}%
,..,i_{n-1},k\right)  .$ Given $\eta\in\left(  0,1\right)  \,$and a
sufficiently large renormalization scale $M>0,$ let\ $\gamma_{M}%
=\{n=x_{1},..,x_{m\left(  k\right)  }=k\}$ be the $M$\emph{-skeleton} of
$\gamma$ (\cite{CIV} section 2.2).

If, for any $t\in\partial\mathbf{K}^{p}.$
\begin{equation}
S_{t}\left(  x\right)  :=\xi_{p}\left(  x\right)  -\left(  t,x\right)
\label{defSte}%
\end{equation}
denotes the \emph{surcharge function} in the direction of $t,$ then
\begin{equation}
\mathcal{C}_{\eta}\left(  t\right)  =\{x\in\mathbb{R}^{d}:S_{t}\left(
x\right)  \leq\eta\xi_{p}\left(  x\right)  \}.
\end{equation}
Let us define
\begin{equation}
\mathbb{B}_{\eta}^{t}\left(  \gamma_{M}\right)  :=\{2\leq l\leq m\left(
k\right)  :x_{l-1}-x_{l}\notin\mathcal{C}_{\eta}\left(  t\right)  \},
\end{equation}
where, for $l\in\mathbb{B}_{\eta}^{t}\left(  \gamma_{M}\right)  ,\,x_{l-1}%
-x_{l}$ are the increments of the path $\gamma_{M}$ backtracking with respect
to $\mathcal{C}_{\eta}\left(  t\right)  .$

Notice that, by (\ref{defSte}), if $l\in\mathbb{B}_{\eta}^{t}\left(
\gamma_{M}\right)  ,$ then
\begin{equation}
S_{t}\left(  x_{l-1}-x_{l}\right)  \geq\eta M.\label{stimbt}%
\end{equation}

\begin{definition}
We call $x_{i}\in\gamma_{M},\,i=2,..,m\left(  k\right)  ,$ a $\left(
t,\eta\right)  $\emph{-good point} of $\gamma_{M},$ if
\begin{equation}
\gamma_{M}\cap\left(  x_{i}+\mathcal{C}_{\eta}\left(  t\right)  \right)
=\{x_{i},..,x_{1}\}
\end{equation}
and denote by $\mathcal{G}_{\eta}^{t}\left(  \gamma_{M}\right)  $ the set of
$\left(  t,\eta\right)  $\emph{-good points} of $\gamma_{M}.$
\end{definition}

We remark that $\mathcal{G}_{\eta}^{t}\left(  \gamma_{M}\right)  $ is a
totally ordered set with respect to the scalar product with $t$ and choose the
same ordering of the set of the $\left(  \eta,K,t\right)  $-break points given
in Definition \ref{defeKbrpt} that is, given $x_{j},x_{l}\in\mathcal{G}_{\eta
}^{t}\left(  \gamma_{M}\right)  ,\,x_{j}>x_{l}$ if $\left(  x_{j}-n,t\right)
<\left(  x_{l}-n,t\right)  .$

\begin{definition}
Let us set
\begin{equation}
\mathcal{B}_{\eta}^{t}\left(  \gamma_{M}\right)  :=\bigvee_{i\geq1}%
\{l_{i},...,r_{i}-1\},
\end{equation}
where
\begin{align}
l_{1} &  :=\max\{j\geq1:x_{j}\notin\mathcal{G}_{\eta}^{t}\left(  \gamma
_{M}\right)  \},\\
r_{1} &  :=\max\{1\leq j<l_{1}:x_{j}-x_{l_{1}}\notin\mathcal{C}_{\eta}\left(
t\right)  \},\\
l_{i} &  :=\max\{1\leq j\leq r_{i-1}:x_{j}\notin\mathcal{G}_{\eta}^{t}\left(
\gamma_{M}\right)  \},\\
r_{i} &  :=\max\{1\leq j<l_{i}:x_{j}-x_{l_{i}}\notin\mathcal{C}_{\eta}\left(
t\right)  \}
\end{align}
and denote by
\begin{equation}
\mathtt{x}\left(  \mathcal{B}_{\eta}^{t}\left(  \gamma_{M}\right)  \right)
:=\{x_{j}\in\gamma_{M}:j\in\mathcal{B}_{\eta}^{t}\left(  \gamma_{M}\right)
\},
\end{equation}
the set of $\left(  t,\eta\right)  $\emph{-bad points} of $\gamma_{M}.$
\end{definition}

We remark that, proceeding as in the proof of \cite{CIV} Lemma 2.2, for any
$\gamma_{M}=\{x_{1},..,x_{m}\},$ by (\ref{stimbt}) we obtain
\begin{equation}
\sum_{k}\sum_{j=l_{k}+1}^{r_{k}}S_{t}\left(  x_{j-1}-x_{j}\right)  \geq
c_{6}\eta M\left|  \mathcal{B}_{\eta}^{t}\left(  \gamma_{M}\right)  \right|
,\label{lbS}%
\end{equation}
with $c_{6}$ a positive constant.

Let $\mathbf{n\in}\mathbb{Z}^{3d},\,\mathbf{C}_{\mathbf{n}}:=\mathbf{C}%
_{\{n_{1},n_{2},n_{3}\}}$ and $k$ be connected to $n_{1},n_{2},n_{3}.$ Then,
following \cite{CI} subsection 2.4, we introduce the $M$\emph{-tree skeleton}
$\Gamma_{\mathbf{n}}^{M}$ of $\mathbf{C}_{\mathbf{n}}=\mathbf{C}%
_{\{k,n_{1},n_{2},n_{3}\}},$ such that
\begin{equation}
\Gamma_{\mathbf{n}}^{M}:=\bigcup_{i=1,2,3}\gamma_{M}^{i}\bigvee L_{M},
\end{equation}
where:

\begin{itemize}
\item  for $i=1,2,3,\,\gamma_{M}^{i}$ is the \emph{self-avoiding trunk} of
$\Gamma_{\mathbf{n}}^{M}$ in the direction $t_{i},$ dual to $n_{i}-k,$ defined
as in subsection 2.4 of \cite{CI}. On the other hand, if there exist three
disjoint self-avoiding open paths $\gamma_{1},\gamma_{2},\gamma_{3}, $
connecting $k$ to $n_{1},n_{2},n_{3},$ we can always choose the self-avoiding
trunks $\gamma_{M}^{1},\gamma_{M}^{2},\gamma_{M}^{3}$ to be the $M$-skeletons
of these paths. In this case, by construction, $\cap_{i=1,2,3}\gamma_{M}^{i}=\{k\}.$

\item $L_{M}$ is the set of \emph{leaves} of $\Gamma_{\mathbf{n}}^{M},$ i.e.
the set of those points of $\Gamma_{\mathbf{n}}^{M}$ which do not belong to
any of the self-avoiding trunks $\gamma_{M}^{1},\gamma_{M}^{2},\gamma_{M}%
^{3},$ defined by means of the construction given below.
\end{itemize}

Let us set $\mathbf{C}_{\mathbf{n}}^{M}:=\bigcup_{y\in\Gamma_{\mathbf{n}}^{M}%
}M\mathbf{U}^{p}\left(  y\right)  $ and, for $i=1,2,3,\,\mathbf{C}%
_{\{k,n_{i}\}}^{M}:=\bigcup_{y\in\Gamma_{i}^{M}}M\mathbf{U}^{p}\left(
y\right)  ,$ where $\Gamma_{i}^{M}:=\gamma_{M}^{i}%
{\textstyle\bigvee}
L_{M}^{i}$ with $L_{M}^{i}$ the set of leaves attached to the trunk
$\gamma_{M}^{i}.$

We say that $\mathbf{C}_{\mathbf{n}}$ is compatible with $\Gamma_{\mathbf{n}%
}^{M},$ and denote this fact by $\mathbf{C}_{\mathbf{n}}\sim\Gamma
_{\mathbf{n}}^{M},$ if $\Gamma_{\mathbf{n}}^{M}$ is the $M$-tree skeleton of
$\mathbf{C}_{\mathbf{n}},$ that is, if for any $m\in\mathbf{C}_{\mathbf{n}}$,
there exits $y\in\Gamma_{\mathbf{n}}^{M}$ such that $m\in M\mathbf{U}%
^{p}\left(  y\right)  .$ Furthermore, since $\Gamma_{\mathbf{n}}^{M}%
=\bigcup_{i=1,2,3}\Gamma_{i}^{M},$ from the compatibility relation
$\mathbf{C}_{\mathbf{n}}\sim\Gamma_{\mathbf{n}}^{M}$ follows the compatibility
relation $\mathbf{C}_{\mathbf{n}}\sim\Gamma_{i}^{M},\,i=1,2,3.$

The construction of $\Gamma_{\mathbf{n}}^{M}$ is similar the one described in
section 2.4 of \cite{CI} and can be carried out algoritmically.

\begin{itemize}
\item [\textbf{step 0}]Define $\Gamma_{\mathbf{n}}^{M}=\bigcup_{i=1,2,3}%
\gamma_{M}^{i}$ and accordingly $\mathbf{C}_{\mathbf{n}}^{M}.\,$Set $i:=1.$

\item[\textbf{step 1}] Define $\Gamma_{i}^{M}=\gamma_{M}^{i}=\{x_{1}%
=n_{i},..,x_{m_{i}}=k\}$ and accordingly $\mathbf{C}_{\{k,n_{i}\}}^{M}.$

\begin{itemize}
\item  If $\forall y\in\mathbf{C}_{\mathbf{n}}\backslash\left(  \left(
\mathbf{C}_{\mathbf{n}}^{M}\backslash\mathbf{C}_{\{k,n_{i}\}}^{M}\right)
\cap\mathbf{C}_{\mathbf{n}}\right)  ,\,\min_{z\in\mathbf{C}_{\{k,n_{i}\}}^{M}%
}\xi_{p}\left(  y-z\right)  \leq M$, then go to step $l_{i}+1.$

\item  Otherwise, preceed to the following update step.
\end{itemize}

\item[\textbf{step 2}] \emph{(update step)} Reorder the points of $\Gamma
_{i}^{M}=\{y_{1},..,y_{l_{i}}\}$ according to lexicographical order starting
from $y_{1}=n_{i}.$ Denoting by $l_{i}$ the cardinality of $\Gamma_{i}^{M},$
set $j:=1.$

\item[\textbf{step }$j$\textbf{\ }$\left(  j\leq l_{i}\right)  $] Screen the
lattice points $y\notin M\left(  \mathbf{U}^{p}\left(  y_{j}\right)
\backslash\partial\mathbf{U}^{p}\left(  y_{j}\right)  \right)  $ which are
endpoints of the edges crossing $M\partial\mathbf{U}^{p}\left(  y_{j}\right)
$ in the lexicographical order and denote their collection by $M\bar{\partial
}\mathbf{U}^{p}\left(  y_{j}\right)  .$

\begin{itemize}
\item  If there exists $y\in M\bar{\partial}\mathbf{U}^{p}\left(
y_{j}\right)  $ such that one can find a self-avoiding open path $\gamma_{y}$
leading from $y$ to $M\partial\mathbf{U}^{p}\left(  y\right)  $ inside
$\mathbb{Z}^{d}\backslash\mathbf{C}_{\mathbf{n}}^{M},$ then set
\begin{align}
L_{M}^{i} &  :=L_{M}^{i}\cup\{y\},\quad\Gamma_{i}^{M}:=\left(  L_{M}^{i}%
\cup\{y\}\right)
{\textstyle\bigvee}
\gamma_{M}^{i},\\
\mathbf{C}_{\{k,n_{i}\}}^{M} &  :=\mathbf{C}_{\{k,n_{i}\}}^{M}\cup
M\mathbf{U}^{p}\left(  y\right)  ,\quad\mathbf{C}_{\mathbf{n}}^{M}%
:=\mathbf{C}_{\mathbf{n}}^{M}\cup M\mathbf{U}^{p}\left(  y\right)
\end{align}
and go back to the update step.

\item  Otherwise, set $j:=j+1$ and go to step $j.$

\item[\textbf{step }$l_{i}+1$] Set $i:=i+1.$

\item  If $i=4,$ then stop.

\item  Otherwise, go to step 1.
\end{itemize}
\end{itemize}

By construction, $L_{M}=%
{\textstyle\bigvee_{i=1,2,3}}
L_{M}^{i}.$

We now define, for $R\in\mathbb{N}$ sufficiently large,
\begin{align}
j_{0} &  :=\min\left\{  j\geq1:x_{j}\in\mathcal{G}_{\eta}^{t_{i}}\left(
\gamma_{M}^{i}\right)  \right\}  \\
j_{l+1} &  :=\min\left\{  j\geq j_{l}:\left|  \left|  x_{j}-x_{j_{l}}\right|
\right|  \geq RM\,,\,x_{j}\in\mathcal{G}_{\eta}^{t_{i}}\left(  \gamma_{M}%
^{i}\right)  \right\}  \quad l\geq0
\end{align}
and consequently
\begin{align}
L_{i}^{M\,bad} &  :=\left\{  y\in L_{M}^{i}:y\notin\bigcup_{l\geq1}\left\{
\left\{  RM\mathbf{U}^{p}\left(  x_{j_{l}}\right)  +\mathcal{C}_{\eta}\left(
t_{i}\right)  \right\}  \cap\mathcal{S}_{\{x_{j_{l}},x_{j_{l-1}}\}}^{t_{i}%
}\right\}  \right\}  \\
\Gamma_{i}^{M\,bad} &  :=L_{i}^{M\,bad}\bigvee\mathtt{x}\left(  \mathcal{B}%
_{\eta}^{t_{i}}\left(  \gamma_{M}^{i}\right)  \right)  ,\quad\Gamma
_{i}^{M\,good}=\Gamma_{i}^{M}\backslash\Gamma_{i}^{M\,bad}\,,\\
\mathbf{C}_{\{k,n_{i}\}}^{M\,bad} &  :=\bigcup_{y\in\Gamma_{i}^{M\,bad}%
}M\mathbf{U}^{p}\left(  y\right)  ,\quad\mathbf{C}_{\{k,n_{i}\}}%
^{M\,good}=\bigcup_{y\in\Gamma_{i}^{M\,good}}M\mathbf{U}^{p}\left(  y\right)
\,.
\end{align}
Moreover,
\begin{equation}
\Gamma_{\mathbf{n}}^{M\,bad}:=\bigcup_{i=1,2,3}\Gamma_{i}^{M\,bad}%
,\quad\mathbf{C}_{\mathbf{n}}^{M\,bad}:=\bigcup_{i=1,2,3}\mathbf{C}%
_{\{k,n_{i}\}}^{M\,bad}\subset\mathbf{C}_{\mathbf{n}}^{M}.
\end{equation}
Hence, $\left|  \Gamma_{i}^{M\,bad}\right|  =\left|  L_{i}^{M\,bad}\right|
+\left|  \mathcal{B}_{\eta}^{t_{i}}\left(  \gamma_{M}^{i}\right)  \right|  $
and $\left|  \Gamma_{\mathbf{n}}^{M\,bad}\right|  \leq\sum_{i=1}^{3}\left|
\Gamma_{i}^{M\,bad}\right|  .$ Proceeding as in \cite{CI} Lemma 2.3, it is
possible to prove that, for any $\delta>0,$ there exists a positive constant
$c_{7}$ such that, for any $i=1,2,3,$
\begin{equation}
\mathbb{P}_{p}\left[  \{k\longleftrightarrow n_{i}\}\cap\left\{  \left|
L_{i}^{M\,bad}\right|  \geq\frac{\delta}{M}\left|  \left|  n_{i}-k\right|
\right|  \right\}  \right]  \leq e^{-\xi_{p}\left(  n_{i}-k\right)
-c_{7}\delta\left|  \left|  n_{i}-k\right|  \right|  },\label{stimbL}%
\end{equation}
Moreover, by (\ref{defSte}) and (\ref{lbS}), argueing as in \cite{CI} Lemma
2.2, we have that, for values of $M$ large enough, there exists a positive
constant $c_{8}$ such that, for any $i=1,2,3,$
\begin{equation}
\mathbb{P}_{p}\left[  \{\gamma_{M}^{i}:k\overset{\gamma_{M}^{i}}%
{\longleftrightarrow n}_{i}\}\cap\left\{  \left|  \mathcal{B}_{\eta}^{t_{i}%
}\left(  \gamma_{M}^{i}\right)  \right|  \geq\frac{\delta}{M}\left|  \left|
n_{i}-k\right|  \right|  \right\}  \right]  \leq e^{-c_{8}\delta\eta\left|
\left|  n_{i}-k\right|  \right|  -\xi_{p}\left(  n_{i}-k\right)
}.\label{stimbp}%
\end{equation}

\begin{definition}
\label{dgood}For any $\delta>0,$ an $M$-tree skeleton $\Gamma_{\mathbf{n}}%
^{M}$ is $\delta$\emph{-good} if, for any $i=1,2,3:$

\begin{enumerate}
\item $\left|  L_{i}^{M\,bad}\right|  \leq\frac{\delta}{M}\left|  \left|
n_{i}-k\right|  \right|  ;$

\item $\left|  \mathcal{B}_{\eta}^{t_{i}}\left(  \gamma_{M}^{i}\right)
\right|  \leq\frac{\delta}{M}\left|  \left|  n_{i}-k\right|  \right|  .$
\end{enumerate}
\end{definition}

Let us define the slabs
\begin{equation}
\mathcal{S}_{l,i}^{M,R}:=\mathcal{S}_{\{k+l4RM\frac{t_{i}}{\left|  \left|
t_{i}\right|  \right|  },k+\left(  l+1\right)  4RM\frac{t_{i}}{\left|  \left|
t_{i}\right|  \right|  }\}}^{t_{i}}\quad i=1,2,3;\,l\in\mathbb{N}.
\end{equation}
For any $i=1,2,3,\,\mathbf{C}_{\mathbf{n}}^{M}$ intersects $\mathcal{N}$
subsequent $\mathcal{S}_{l,i}^{M,R}$ slabs, with $\mathcal{N}\geq\frac{\left(
1-\eta\right)  c_{-}\left(  p\right)  }{4RM}\left|  \left|  n_{i}-k\right|
\right|  .$ Furthermore, if $\Gamma_{\mathbf{n}}^{M}$ is $\delta$-good, at
most $\frac{2\delta\left|  \left|  n_{i}-k\right|  \right|  }{M}\ $of the
$\mathcal{S}_{l,i}^{M,R}$ slabs contain points belonging to $\mathbf{C}%
_{\{k,n_{i}\}}^{M\,bad}.$ Hence, if we choose $\delta\in\left(  0,\frac
{\left(  1-\eta\right)  c_{-}\left(  p\right)  }{16R}\right)  ,$ the number of
$\mathcal{S}_{l,i}^{M,R}$ slabs containing only points of $\mathbf{C}%
_{\{k,n_{i}\}}^{M\,good},$ which we will call $\delta$-good slabs, is larger
than $\frac{1}{2}\frac{\left(  1-\eta\right)  c_{-}\left(  p\right)  }%
{4RM}\left|  \left|  n_{i}-k\right|  \right|  .$ Renumbering all the $\delta
$-good $\mathcal{S}_{l,i}^{M,R}$ slabs as $\mathcal{S}_{l_{1},i}%
^{M,R},..,\mathcal{S}_{l_{r},i}^{M,R},\,r\geq\frac{\left(  1-\eta\right)
c_{-}\left(  p\right)  }{8RM}\left|  \left|  n_{i}-k\right|  \right|  ,\,$for
every $\delta$-good $\mathcal{S}_{l_{j},i}^{M,R}$ slab and every cluster
$\mathbf{C}_{\mathbf{n}}$ compatible with $\Gamma_{\mathbf{n}},$ we have
\begin{equation}
\mathbf{C}_{\mathbf{n}}\cap\mathcal{S}_{l_{j},i}^{M,R}\subseteq\mathcal{R}%
_{l_{j},i}^{M,R}:=\bigcup_{x\in\Gamma_{i}^{M}\cap\mathcal{S}_{l_{j},i}^{M,R}%
}4M\mathbf{U}^{p}\left(  x\right)  .
\end{equation}
Choosing $K$ such that $\frac{K}{M}>1$ and setting
\begin{equation}
\operatorname*{dist}\left(  \mathcal{R}_{l_{j},i}^{M,R};\mathcal{R}%
_{l_{j^{\prime}},i}^{M,R}\right)  :=\min_{\substack{y\in\mathcal{R}_{l_{j}%
,i}^{M,R}\\y^{\prime}\in\mathcal{R}_{l_{j^{\prime}},i}^{M,R}}}\xi_{p}\left(
y-y^{\prime}\right)  \,,
\end{equation}
we define the application $\{1,..,r\}\ni j\longmapsto q\left(  j\right)
=j_{q}\in\{1,..,s\}$ such that
\begin{align}
j_{1} &  :=1\,,\\
j_{q+1} &  :=\min\left\{  j>j_{q}:\operatorname*{dist}\left(  \mathcal{R}%
_{l_{j},i}^{M,R};\mathcal{R}_{l_{j_{q}},i}^{M,R}\right)  \geq\frac{2K}{\eta
}\right\}  .
\end{align}
As in \cite{CI} section 3.3 and 4, for each $i=1,2,3,$ it is possible to
modify at most $c_{9}\left(  RM\right)  ^{d}$ bonds inside the regions
$\mathcal{R}_{l_{j_{1}},i}^{M,R},..,\mathcal{R}_{l_{j_{s}},i}^{M,R}%
,\,s\geq\frac{\eta\left(  1-\eta\right)  c_{-}\left(  p\right)  }%
{16KRM}\left|  \left|  n_{i}-k\right|  \right|  ,$ in an independent way such
that the resulting modified cluster is still compatible with $\Gamma_{i}^{M}$
and contains at least one $t_{i}$-break point located in each these regions.
By construction, these $t_{i}$-break points verify condition 2 of Definition
\ref{defeKbrpt}. Since for any $x\in\mathcal{C}_{\eta}\left(  t\right)  ,$%
\begin{equation}
\left|  \left|  P_{i}^{\perp}x\right|  \right|  \leq\sqrt{\frac{1-\left(
1-\eta\right)  ^{2}c_{-}^{2}\left(  p\right)  }{\left(  1-\eta\right)
^{2}c_{-}^{2}\left(  p\right)  }}\left(  t_{i},x\right)  ,
\end{equation}
we can choose $\frac{K}{M}>4\eta R$ large enough such that at least half the
points in $\mathcal{R}_{l_{j_{q+1}},i}^{M,R}$ belong to $z+\mathcal{C}_{\eta
}\left(  t_{i}\right)  ,$ for every $z\in\mathcal{R}_{l_{j_{q}},i}^{M,R}.$
Therefore, at least half of the $t_{i}$-break points of the modified cluster
satisfy condition 1 of Definition \ref{defeKbrpt}. We also remark that, since
by construction any $t_{i}$-break point of the modified cluster belongs to a
neighborhood $M\mathbf{U}^{p}\left(  x_{j}\right)  $ of some good point
$x_{j},$ one can choose $R>\frac{\left(  1-\eta\right)  c_{-}\left(  p\right)
}{2\sqrt{1-\left(  1-\eta\right)  ^{2}c_{-}^{2}\left(  p\right)  }}$ large
enough such that, for any $q=1,..,s,$ there are at least one $t_{i}$-break
point $b_{q}\in\mathcal{R}_{l_{j_{q}},i}^{M,R}$ and one $t_{i}$-break point
$b_{q+1}\in\mathcal{R}_{l_{j_{q+1}},i}^{M,R}$ which verify conditions 3 of
Definition \ref{defeKbrpt}, provided that the slabs $\mathcal{S}_{l_{j}%
,i}^{M,R},$ for $l_{j_{q}}<l_{j}<l_{j_{q+1}},$\ are $\delta$-good. On the
other hand, if $\mathcal{S}_{l_{j},i}^{M,R}$ contains points of $L_{i}%
^{M\,bad},$ at most only the $t_{i}$-break points of the modified cluster
belonging to $\mathcal{S}_{l_{j_{q+1}},i}^{M,R},..,\mathcal{S}_{l_{j_{q+\rho}%
},i}^{M,R},$ with $\rho\leq\left[  \frac{1}{4}s\right]  ,$ do not verify
condition 3 of \ref{defeKbrpt}.

Therefore, proceding as in the proofs of Lemma 4.1 in \cite{CI}, this
argument, together with the estimates (\ref{stimbL}) (\ref{stimbp}), proves
that for any $\eta>0$ sufficiently small, there exists $\delta_{1}=\delta
_{1}\left(  \eta,p\right)  $ and $c_{10}=c_{10}\left(  \eta,p\right)  >0$ such
that
\begin{equation}
\mathbb{P}_{p}\left[  \left\{  \left|  \mathbf{B}^{t_{i}}\left(  k,n_{i}%
;\eta,K\right)  \right|  <\delta_{1}\left|  \left|  n_{i}-k\right|  \right|
\right\}  \cap\left\{  k\overset{\bar{h}_{t_{i}}}{\longleftrightarrow}%
n_{i}\right\}  \right]  \leq e^{-\xi_{p}\left(  n_{i}-k\right)  -c_{10}\left|
\left|  n_{i}-k\right|  \right|  }.
\end{equation}
Since $\bar{h}_{t_{i}}$ and $\bar{f}_{t_{i}}^{\eta,K}$ satisfy a renewal
equation analogous to (\ref{reneq}), the last inequality implies that $\bar
{f}_{t_{i}}^{\eta,K}$ verifies a \emph{mass-gap} type condition similar to the
one verified by $f_{t_{i}}$ (\cite{CI} section 4), that is there exists a
positive constant $c_{11}=c_{11}\left(  \eta,p\right)  $ such that
\begin{equation}
\frac{\bar{f}_{t_{i}}^{\eta,K}\left(  k-n_{i}\right)  }{\bar{h}_{t_{i}}\left(
k-n_{i}\right)  }\leq e^{-c_{11}\left|  \left|  n_{i}-k\right|  \right|  }\,.
\end{equation}
Hence, since $f_{t_{i}}^{\eta,K}\left(  n_{i}-k\right)  <\bar{f}_{t_{i}}%
^{\eta,K}\left(  n_{i}-k\right)  ,$ a similar estimate holds also for
$f_{t_{i}}^{\eta,K}.$

The BK inequality and the previous construction also imply
\begin{equation}
\mathbb{P}_{p}[F\left(  k;\mathbf{n}\right)  \cap\{\left|  \mathbf{B}^{t_{i}%
}\left(  k,n_{i};\eta,K\right)  \right|  <\delta_{1}\left|  \left|
n_{i}-k\right|  \right|  \}]\leq e^{-\varphi_{p,\mathbf{n}}\left(  k\right)
-c_{10}\left|  \left|  n_{i}-k\right|  \right|  }\quad
i=1,2,3.\label{massgap1}%
\end{equation}

Let us now consider the event
\begin{equation}
G_{\delta_{2}}^{\eta,K}\left(  k;\mathbf{n}\right)  :=F\left(  k;\mathbf{n}%
\right)  \cap\bigcap_{i=1,2,3}\left\{  k\overset{\bar{h}_{t_{i}}%
}{\longleftrightarrow}n_{i}\,;\,\left|  \mathbf{B}^{t_{i}}\left(  k,n_{i}%
;\eta,K\right)  \right|  \leq\delta_{2}\left|  \left|  n_{i}-k\right|
\right|  \right\}
\end{equation}
with $\delta_{2}\leq\delta_{1}.$ To estimate the probability of $G_{\delta
_{2}}^{\eta,K}\left(  k;\mathbf{n}\right)  $ we can repeat the same
renormalization procedure previously set up to prove the mass-gap type
condition for the $f_{t_{i}}^{\eta,K}$\ connections along one direction
$\frac{t_{i}}{\left|  \left|  t_{i}\right|  \right|  },$ except that now we
need to consider all the three directions $\frac{t_{1}}{\left|  \left|
t_{1}\right|  \right|  },\frac{t_{2}}{\left|  \left|  t_{2}\right|  \right|
},\frac{t_{3}}{\left|  \left|  t_{3}\right|  \right|  },$ at once. Given
$M$-tree skeleton $\Gamma_{\mathbf{n}}^{M},$ by the BK inequality it follows
that $\mathbb{P}_{p}[\Gamma_{\mathbf{n}}^{M}]\leq\prod_{i=1}^{3}\mathbb{P}%
_{p}[\Gamma_{i}^{M}].$ Consequently we obtain
\begin{equation}
\mathbb{P}_{p}[G_{\delta_{2}}^{\eta,K}\left(  k;\mathbf{n}\right)  ]\leq
e^{-\varphi_{p,\mathbf{n}}\left(  k\right)  -c_{12}\sum_{i=1}^{3}\left|
\left|  n_{i}-k\right|  \right|  }.\label{massgap2}%
\end{equation}

\subsection{Proof of Theorem \ref{main}}

\begin{definition}
\label{def1}Given $\eta\in\left(  0,1\right)  $ and $K$ sufficiently large,
let $\mathbf{n}\in X_{3}^{\prime}$ and let $\mathbf{t=}\left(  t_{1}%
,t_{2},t_{3}\right)  $ be the vector in $\left(  \mathbb{R}^{d}\right)  ^{3}$
whose entries, $t_{1},t_{2},t_{3}$ are respectively the polar points to
$n_{1}-x_{0}\left(  \mathbf{n}\right)  ,n_{2}-x_{0}\left(  \mathbf{n}\right)
,n_{3}-x_{0}\left(  \mathbf{n}\right)  .$ By (\ref{defi}), for any
$k\in\mathbb{Z}^{d}$ and $i=1,2,3,\,$we denote by $b_{i,\mu_{i}}$ the element
of $\mathbf{B}^{t_{i}}\left(  k,n_{i};\eta,K\right)  $ such that the scalar
product $\left|  \left(  n_{i}-b_{i,\mu_{i}},t_{i}\right)  \right|  $ is
maximal and define $T\left(  \mathbf{b};k,\mathbf{n}\right)  =T\left(
b_{1},b_{2},b_{3};k,\mathbf{n}\right)  $ to be the event that $k$ is connected
to $n_{1},n_{2},n_{3}$ by three self-avoiding disjoint open paths incidents in
$b_{1},b_{2},b_{3},$ these being the positions assumed respectively by the
random points $b_{1,\mu_{1}-1},b_{2,\mu_{2}-1},b_{3,\mu_{3}-1}.$ Moreover, any
configuration $\mathbf{b}\in\left(  \mathbb{Z}^{d}\right)  ^{3}$ for the
$\left(  \eta,K,\mathbf{t}\right)  $-break points $b_{1,\mu_{1}-1}%
,b_{2,\mu_{2}-1},b_{3,\mu_{3}-1}$ will be called \emph{admissible} for $k$ if
$\mathbb{P}_{p}[T\left(  \mathbf{b};k,\mathbf{n}\right)  ]>0.$
\end{definition}

We remark that, given $\mathbf{n}\in X_{3}^{\prime}$ and any $k_{1},k_{2}%
\in\mathbb{Z}^{d},$ if we choose two distinct vectors $\mathbf{b}%
_{1},\mathbf{b}_{2}\in\left(  \mathbb{Z}^{d}\right)  ^{3},$ then $T\left(
\mathbf{b}_{1};k_{1},\mathbf{n}\right)  $ and $T\left(  \mathbf{b}_{2}%
;k_{2},\mathbf{n}\right)  $ are disjoint.

For any $\mathbf{b}\in\left(  \mathbb{Z}^{d}\right)  ^{3}$, let $\mathcal{T}%
\left(  \mathbf{b}\right)  :=\bigcap_{i=1}^{3}\mathcal{H}_{b_{i}}^{t_{i},-}.$
Then, from the previous definition, it follows that $\mathbf{b}$ cannot be
admissible for $k\in\mathbb{Z}^{d}$ if $k\notin\mathcal{T}\left(
\mathbf{b}\right)  .$ If $k_{1},k_{2}\in\mathbb{Z}^{d}$ and $\mathbf{b}$ is
admissible for both $k_{1}$ and $k_{2},$ then $T\left(  \mathbf{b}%
;k_{1},\mathbf{n}\right)  $ and $T\left(  \mathbf{b};k_{2},\mathbf{n}\right)
$ need not be disjoint.

Therefore, the event $F\left(  k;\mathbf{n}\right)  $ allows the
decomposition
\begin{equation}
F\left(  k;\mathbf{n}\right)  =\bigvee_{b_{1},b_{2},b_{3}\in\mathbb{Z}^{d}%
}T\left(  b_{1},b_{2},b_{3};k,\mathbf{n}\right)  \bigvee T^{\ast}\left(
k;\mathbf{n}\right)  ,\label{decF}%
\end{equation}
with $T^{\ast}\left(  k;\mathbf{n}\right)  =F\left(  k;\mathbf{n}\right)
\cap\bigcup_{i=1,2,3}\left\{  \left|  \mathbf{B}^{t_{i}}\left(  k,n_{i}%
;\eta,K\right)  \right|  \leq1\right\}  .$

Before entering ito details, let us describe the main ideas of the proof of
Theorem (\ref{main}).

As a first step, in the following proposition, we derive the asymptotic
behaviour for the probability of the event $F\left(  [Nx_{0}\left(
\mathbf{x}\right)  +\sqrt{N}y];\left[  N\mathbf{x}\right]  \right)  $ that the
point $[Nx_{0}\left(  \mathbf{x}\right)  +\sqrt{N}y]$ is connected by three
disjoint self-avoiding open paths to the points $[Nx_{1}],[Nx_{2}],[Nx_{3}%
],$\ as $N$ goes to infinity. The event $F\left(  [Nx_{0}\left(
\mathbf{x}\right)  +\sqrt{N}y];\left[  N\mathbf{x}\right]  \right)  ,$ apart
from terms that can be neglected, can be decomposed into a partition according
to the positions of the points $b_{1},b_{2},b_{3},$ as shown in (\ref{decF}),
where the distances of $b_{i}$'s from $[Nx_{0}\left(  \mathbf{x}\right)
+\sqrt{N}y]$ can be assumed to be smaller than $N^{\beta}$ with $0<\beta
<\frac{1}{2}.$ The definition of the $b_{i}$'s implies that the probability of
a term of such a decomposition is of the form
\begin{equation}
g_{p}^{\eta,K}\left(  k;\mathbf{b}\right)  \prod_{i=1,2,3}\tilde{h}_{t_{i}%
}^{\eta,K}\left(  [Nx_{i}]-b_{i}\right)  ,
\end{equation}
with $k=[Nx_{0}\left(  \mathbf{x}\right)  +\sqrt{N}y],$ where $g_{p}^{\eta
,K}\left(  k;\mathbf{b}\right)  $ is a translationally invariant function. The
desired asymptotics follows then from the Ornstein-Zernike estimate
(\ref{ashtt}) for $\tilde{h}_{t}^{\eta,K}$ and the expansion of the function
$\xi_{p}$ that appears in it.

The second step is to obtain the asymptotic behaviour of the probability of
the event $E\left(  [N\mathbf{x}]\right)  $ that appears in the denominator of
the conditional expectation (\ref{PFE}) as $N$ tends to infinity.\thinspace As
before, apart from terms that can be neglected, we can decompose this event
into a partition according to the positions of the points $b_{1},b_{2},b_{3}.$
The probability of a term of such a decomposition can be written as
\begin{equation}
g_{p}^{\eta,K}\left(  \mathbf{b}\right)  \prod_{i=1,2,3}\tilde{h}_{t_{i}%
}^{\eta,K}\left(  [Nx_{i}]-b_{i}\right)  ,
\end{equation}
where $g_{p}^{\eta,K}\left(  \mathbf{b}\right)  $\ is translation invariant.
It can be assumed, neglecting terms of small probability, that the distances
among the $b_{i}$'s are smaller than $N^{\beta},$ with $\beta\in\left(
0,\frac{1}{2}\right)  ,$ and that the distance of $x_{0}\left(  [N\mathbf{x]}%
\right)  $ from each of the $b_{i}$'s are smaller than $N^{\alpha},$ with
$\alpha\in\left(  \frac{1}{2},1\right)  .$ The Ornstein-Zernike estimate
(\ref{ashtt}) for $\tilde{h}_{t}^{\eta,K}$ and the expansion of the function
$\xi_{p}$ that appears in it give then the desired asymptotics.

\begin{proposition}
For $\mathbf{x}\in X_{3}^{\prime},$ let $x_{0}=x_{0}\left(  \mathbf{x}\right)
$ be the unique minimizer of $\varphi_{p,\mathbf{x}}.$ Then, for any
$y\in\mathbb{R}^{d},$ there exists a positive real analytic function
$\Theta_{p}$ on $X_{3}^{\prime}$ such that, for $d\geq2$ and $p<p_{c}\left(
d\right)  ,$
\begin{equation}
\mathbb{P}_{p}\left[  F\left(  [Nx_{0}\left(  \mathbf{x}\right)  +\sqrt
{N}y];\left[  N\mathbf{x}\right]  \right)  \right]  =\frac{\Theta_{p}\left(
\mathbf{x}\right)  }{\left(  2\pi N^{d-1}\right)  ^{\frac{3}{2}}}%
e^{-\varphi_{p,[N\mathbf{x}]}\left(  x_{0}\left(  [N\mathbf{x}]\right)
\right)  -\frac{\left(  y,H_{\varphi}\left(  x_{0}\left(  \mathbf{x}\right)
,\mathbf{x};p\right)  y\right)  }{2}}\left(  1+o\left(  1\right)  \right)
.\label{lbPF}%
\end{equation}
\end{proposition}

\begin{proof}
Let $k\in\mathbb{Z}^{d}$ be such that $\left|  \left|  k-x_{0}\left(
[N\mathbf{x}]\right)  \right|  \right|  \leq c_{13}N^{\frac{1}{2}}$ and denote
by $t_{i}=t_{i}\left(  \mathbf{x}\right)  $ the polar point to $x_{0}\left(
\mathbf{x}\right)  -x_{i},\,i=1,2,3.$ We can choose $\eta\in\left(
0,1\right)  $ small enough such that\ $\mathcal{C}_{2\eta}\left(
t_{i}\right)  \cap\mathcal{C}_{2\eta}\left(  t_{j}\right)  =0,\,i\neq j=1,2,3,
$ then, from Lemma \ref{noint}, it follows that for\ $i=1,2,3,\,\mathbf{C}%
_{\{b_{i},[Nx_{i}]\}}\cap\mathcal{H}_{b_{i,\mu_{i}}}^{t_{i},+}\subset
b_{i,\mu_{i}}+\mathcal{C}_{2\eta}\left(  t_{i}\right)  $ and for $i\neq j$
\begin{equation}
\left(  b_{i,\mu_{i}}+\mathcal{C}_{2\eta}\left(  t_{i}\right)  \right)
\cap\left(  b_{j,\mu_{j}}+\mathcal{C}_{2\eta}\left(  t_{j}\right)  \right)
=\emptyset.
\end{equation}

If, for any $\mathbf{b}\in\left(  \mathbb{Z}^{d}\right)  ^{3}$ and $K$
sufficiently large, we define the function
\begin{equation}
g_{p}^{\eta,K}\left(  k;\mathbf{b}\right)  =g_{p}^{\eta,K}\left(
k;b_{1},b_{2},b_{3}\right)  :=\mathbb{P}_{p}[G^{\eta,K}\left(  k;b_{1}%
,b_{2},b_{3}\right)  ],\label{defgeK}%
\end{equation}
which is the probability of the event
\begin{equation}
G^{\eta,K}\left(  k;b_{1},b_{2},b_{3}\right)  =G^{\eta,K}\left(
k;\mathbf{b}\right)  :=F\left(  k;\mathbf{b}\right)  \cap\bigcap_{i=1}%
^{3}\left\{  k\overset{\bar{h}_{t_{i}}}{\longleftrightarrow b_{i}}\,,\,\left|
\mathbf{B}^{t_{i}}\left(  k,b_{i};\eta,K\right)  \right|  =1\right\}  ,
\end{equation}
then, by Definition \ref{def1}, we have
\begin{equation}
\mathbb{P}_{p}[T\left(  k;\mathbf{b,[}N\mathbf{x]}\right)  ]:=g_{p}^{\eta
,K}\left(  k;\mathbf{b}\right)  \prod_{i=1,2,3}\tilde{h}_{t_{i}}^{\eta
,K}\left(  [Nx_{i}]-b_{i}\right)  .
\end{equation}
Notice that $g_{p}^{\eta,K}\left(  k;\mathbf{b}\right)  $ is translation
invariant, i.e.
\begin{equation}
g_{p}^{\eta,K}\left(  k+u;b_{1}+u,b_{2}+u,b_{3}+u\right)  =g_{p}^{\eta
,K}\left(  k;b_{1},b_{2},b_{3}\right)  \,,\quad u\in\mathbb{Z}^{d}\,.
\end{equation}

By (\ref{decF}), we have
\begin{equation}
\mathbb{P}_{p}[F\left(  k;\left[  N\mathbf{x}\right]  \right)  ]=\sum
_{b_{1},b_{2},b_{3}\in\mathbb{Z}^{d}}\mathbb{P}_{p}[T\left(  b_{1},b_{2}%
,b_{3};k,[N\mathbf{x]}\right)  ]+\mathbb{P}_{p}[T^{\ast}\left(  k;[N\mathbf{x}%
]\right)  ].
\end{equation}
Since by (\ref{massgap1}),
\begin{equation}
\mathbb{P}_{p}[T^{\ast}\left(  k;[N\mathbf{x}]\right)  ]\leq\sum
_{i=1,2,3}e^{-\varphi_{p,[N\mathbf{x]}}\left(  k\right)  -c_{10}\left|
\left|  [Nx_{i}]-k\right|  \right|  },\label{stimPT*}%
\end{equation}
then, by (\ref{defgeK}), we need to estimate
\begin{align}
&  \sum_{b_{1},b_{2},b_{3}\in\mathbb{Z}^{d}}\mathbb{P}_{p}[T\left(
b_{1},b_{2},b_{3};k,[N\mathbf{x]}\right)  ]\\
&  =\sum_{b_{1},b_{2},b_{3}\in\mathbb{Z}^{d}}g_{p}^{\eta,K}\left(
k;b_{1},b_{2},b_{3}\right)  \prod_{i=1,2,3}\tilde{h}_{t_{i}}^{\eta,K}\left(
[Nx_{i}]-b_{i}\right) \nonumber\\
&  =e^{-\varphi_{p,[N\mathbf{x]}}\left(  k\right)  }\sum_{b_{1},b_{2},b_{3}%
\in\mathbb{Z}^{d}}g_{p}^{\eta,K}\left(  k;b_{1},b_{2},b_{3}\right)
\prod_{i=1,2,3}e^{\xi_{p}\left(  [Nx_{i}]-k\right)  }\tilde{h}_{t_{i}}%
^{\eta,K}\left(  [Nx_{i}]-b_{i}\right)  .\nonumber
\end{align}
By (\ref{massgap2}), it follows that
\begin{equation}
g_{p}^{\eta,K}\left(  k;b_{1},b_{2},b_{3}\right)  \leq e^{-\varphi
_{p,\mathbf{b}}\left(  k\right)  -c_{11}\sum_{i=1,2,3}\left|  \left|
b_{i}-k\right|  \right|  }\label{stimghK}%
\end{equation}
But, by the convexity of $\xi_{p},$ there exists a positive constant $c_{14} $
such that, $\forall\beta\in\left(  0,1\right)  $ and $N$ large enough, by
(\ref{stimghK}), (\ref{f<exi}) and (\ref{lbxi}), we have
\begin{align}
&  \sum_{b_{1}\in\mathbb{Z}^{d}\,:\,\left|  \left|  b_{1}-k\right|  \right|
>N^{\beta}}\sum_{b_{2},b_{3}\in\mathbb{Z}^{d}}g_{p}^{\eta,K}\left(
k;b_{1},b_{2},b_{3}\right)  \prod_{i=1}^{3}e^{\xi_{p}\left(  [Nx_{i}%
]-k\right)  }\tilde{h}_{t_{i}}^{\eta,K}\left(  [Nx_{i}]-b_{i}\right)
\label{stim1}\\
&  \leq\sum_{b_{1}\in\mathbb{Z}^{d}\,:\,\left|  \left|  b_{1}-k\right|
\right|  >N^{\beta}}e^{-c_{11}\left|  \left|  b_{1}-k\right|  \right|
+\xi_{p}\left(  [Nx_{1}]-k\right)  -\xi_{p}\left(  [Nx_{1}]-b_{1}\right)
-\xi_{p}\left(  b_{1}-k\right)  }\times\nonumber\\
&  \times\prod_{i=1}^{2}\sum_{b_{i}\in\mathbb{Z}^{d}}e^{-c_{11}\left|  \left|
b_{i}-k\right|  \right|  +\xi_{p}\left(  [Nx_{i}]-k\right)  -\xi_{p}\left(
[Nx_{i}]-b_{i}\right)  -\xi_{p}\left(  b_{i}-k\right)  }\leq e^{-c_{14}%
N^{\beta}}\nonumber
\end{align}
and analogous estimates hold for the sums over $b_{2}$ and $b_{3}.$ Thus, by
(\ref{ashtt}), we are left with the estimate of
\begin{align}
&  \sum_{\substack{b_{1}\in\mathbb{Z}^{d}\,:\,\left|  \left|  b_{1}-k\right|
\right|  \leq N^{\beta} \\b_{2}\in\mathbb{Z}^{d}\,:\,\left|  \left|
b_{2}-k\right|  \right|  \leq N^{\beta} \\b_{3}\in\mathbb{Z}^{d}\,:\,\left|
\left|  b_{3}-k\right|  \right|  \leq N^{\beta}}}g_{p}^{\eta,K}\left(
k;b_{1},b_{2},b_{3}\right)  \prod_{i=1,2,3}e^{\xi_{p}\left(  [Nx_{i}%
]-k\right)  }\tilde{h}_{t_{i}}^{\eta,K}\left(  [Nx_{i}]-b_{i}\right)
\label{stimteta}\\
&  =\sum_{\substack{b_{1}\in\mathbb{Z}^{d}\,:\,\left|  \left|  b_{1}-k\right|
\right|  \leq N^{\beta} \\b_{2}\in\mathbb{Z}^{d}\,:\,\left|  \left|
b_{2}-k\right|  \right|  \leq N^{\beta} \\b_{3}\in\mathbb{Z}^{d}\,:\,\left|
\left|  b_{3}-k\right|  \right|  \leq N^{\beta}}}g_{p}^{\eta,K}\left(
k;b_{1},b_{2},b_{3}\right)  \prod_{i=1,2,3}\frac{\tilde{\Lambda}_{p}\left(
\frac{Nx_{i}-b_{i}}{\left|  \left|  Nx_{i}-b_{i}\right|  \right|  }%
,t_{i}\right)  }{\sqrt{2\pi N^{d-1}\left|  \left|  x_{i}-\frac{b_{i}}%
{N}\right|  \right|  ^{d-1}}}\times\nonumber\\
&  \times\exp[\xi_{p}\left(  [Nx_{i}]-k\right)  -\xi_{p}\left(  [Nx_{i}%
]-b_{i}\right)  ]\left(  1+o\left(  1\right)  \right)  .\nonumber
\end{align}
By the convexity of $\xi_{p}$ and by (\ref{stimghK}), we have
\begin{align}
&  \sum_{\substack{b_{1}\in\mathbb{Z}^{d}\,:\,\left|  \left|  b_{1}-k\right|
\right|  \leq N^{\beta} \\b_{2}\in\mathbb{Z}^{d}\,:\,\left|  \left|
b_{2}-k\right|  \right|  \leq N^{\beta} \\b_{3}\in\mathbb{Z}^{d}\,:\,\left|
\left|  b_{3}-k\right|  \right|  \leq N^{\beta}}}g_{p}^{\eta,K}\left(
k;b_{1},b_{2},b_{3}\right)  \prod_{i=1,2,3}\exp[\xi_{p}\left(  [Nx_{i}%
]-k\right)  -\xi_{p}\left(  [Nx_{i}]-b_{i}\right)  ]\label{stimghK2}\\
&  \leq\left(  \sum_{\substack{b_{1}\in\mathbb{Z}^{d}\,:\,\left|  \left|
b_{1}-k\right|  \right|  \leq N^{\beta}}}e^{-c_{12}\left|  \left|
b_{1}-k\right|  \right|  }\right)  ^{3}.\nonumber
\end{align}
Moreover, by translation invariance, setting $\forall i=1,2,3,\,b_{i}%
=a_{i}+k,$ for $\beta\in\left(  0,\frac{1}{2}\right)  ,$
\begin{align}
&  \sum_{\substack{b_{1}\in\mathbb{Z}^{d}\,:\,\left|  \left|  b_{1}-k\right|
\right|  \leq N^{\beta} \\b_{2}\in\mathbb{Z}^{d}\,:\,\left|  \left|
b_{2}-k\right|  \right|  \leq N^{\beta} \\b_{3}\in\mathbb{Z}^{d}\,:\,\left|
\left|  b_{3}-k\right|  \right|  \leq N^{\beta}}}g_{p}^{\eta,K}\left(
k;b_{1},b_{2},b_{3}\right)  \prod_{i=1,2,3}\frac{\tilde{\Lambda}_{p}\left(
\frac{Nx_{i}-b_{i}}{\left|  \left|  Nx_{i}-b_{i}\right|  \right|  }%
,t_{i}\right)  }{\sqrt{2\pi N^{d-1}\left|  \left|  x_{i}-\frac{b_{i}}%
{N}\right|  \right|  ^{d-1}}}\times\label{stimPFb}\\
&  \times\exp[\xi_{p}\left(  [Nx_{i}]-k\right)  -\xi_{p}\left(  [Nx_{i}%
]-b_{i}\right)  ]\left(  1+o\left(  1\right)  \right) \nonumber\\
&  =\prod_{i=1,2,3}\frac{\tilde{\Lambda}_{p}\left(  \frac{x_{i}-x_{0}\left(
\mathbf{x}\right)  }{\left|  \left|  x_{i}-x_{0}\left(  \mathbf{x}\right)
\right|  \right|  },t_{i}\right)  }{\sqrt{2\pi N^{d-1}\left|  \left|
x_{i}-x_{0}\left(  \mathbf{x}\right)  \right|  \right|  ^{d-1}}}%
\sum_{\substack{a_{1}\in\mathbb{Z}^{d}\,:\,\left|  \left|  a_{1}\right|
\right|  \leq N^{\beta} \\a_{2}\in\mathbb{Z}^{d}\,:\,\left|  \left|
a_{2}\right|  \right|  \leq N^{\beta} \\a_{3}\in\mathbb{Z}^{d}\,:\,\left|
\left|  a_{3}\right|  \right|  \leq N^{\beta}}}g_{p}^{\eta,K}\left(
0;a_{1},a_{2},a_{3}\right)  \times\nonumber\\
&  \times\exp[\xi_{p}\left(  [Nx_{i}]-k\right)  -\xi_{p}\left(  [Nx_{i}%
]-a_{i}-k\right)  ]\left(  1+o\left(  1\right)  \right)  .\nonumber
\end{align}
But
\begin{align}
\xi_{p}\left(  [Nx_{i}]-k\right)   &  =\xi_{p}\left(  [Nx_{i}]-x_{0}%
([N\mathbf{x}])\right)  -\left(  \nabla\xi_{p}\left(  [Nx_{i}]-x_{0}%
([N\mathbf{x}])\right)  ,k-x_{0}([N\mathbf{x}])\right)  +\\
&  +\frac{1}{2}\left(  k-x_{0}([N\mathbf{x}]),H_{\xi}\left(  [Nx_{i}%
]-x_{0}([N\mathbf{x}]);p\right)  \left(  k-x_{0}([N\mathbf{x}])\right)
\right)  +O\left(  \frac{1}{\sqrt{N}}\right) \nonumber\\
\xi_{p}\left(  [Nx_{i}]-a_{i}-k\right)   &  =\xi_{p}\left(  [Nx_{i}%
]-x_{0}([N\mathbf{x}])\right)  -\left(  \nabla\xi_{p}\left(  [Nx_{i}%
]-x_{0}([N\mathbf{x}])\right)  ,a_{i}+k-x_{0}([N\mathbf{x}])\right)
+\nonumber\\
&  +\frac{1}{2}\left(  a_{i}+k-x_{0}([N\mathbf{x}]),H_{\xi}\left(
[Nx_{i}]-x_{0}([N\mathbf{x}]);p\right)  \left(  a_{i}+k-x_{0}([N\mathbf{x}%
])\right)  \right)  +O\left(  \frac{1}{\sqrt{N}}\right) \nonumber
\end{align}
Hence, since $H_{\xi}\left(  \cdot;p\right)  $ is a homogeneous function of
order $-1$ in $\mathbb{R}^{d}\backslash\{0\},$ (\ref{stimPFb}) is equal to
\begin{align}
&  \prod_{i=1,2,3}\frac{\tilde{\Lambda}_{p}\left(  \frac{x_{i}-x_{0}\left(
\mathbf{x}\right)  }{\left|  \left|  x_{i}-x_{0}\left(  \mathbf{x}\right)
\right|  \right|  },t_{i}\right)  }{\sqrt{2\pi N^{d-1}\left|  \left|
x_{i}-x_{0}\left(  \mathbf{x}\right)  \right|  \right|  ^{d-1}}}%
\sum_{\substack{a_{1}\in\mathbb{Z}^{d}\,:\,\left|  \left|  a_{1}\right|
\right|  \leq N^{\beta} \\a_{2}\in\mathbb{Z}^{d}\,:\,\left|  \left|
a_{2}\right|  \right|  \leq N^{\beta} \\a_{3}\in\mathbb{Z}^{d}\,:\,\left|
\left|  a_{3}\right|  \right|  \leq N^{\beta}}}g_{p}^{\eta,K}\left(
0;a_{1},a_{2},a_{3}\right)  \times\\
&  \times\exp[\sum_{i=1,2,3}\left(  \nabla\xi_{p}\left(  [Nx_{i}%
]-x_{0}([N\mathbf{x}])\right)  ,a_{i}\right)  ]\left(  1+o\left(  1\right)
\right)  .\nonumber
\end{align}
Since there exists a constant $c_{15}$ such that $\left|  \left|  \nabla
\xi_{p}\left(  [Nx_{i}]-x_{0}([N\mathbf{x}])\right)  -t_{i}\right|  \right|
\leq\frac{c_{15}}{N},$ then, by (\ref{stimghK2}),
\begin{align}
&  \sum_{\substack{a_{1}\in\mathbb{Z}^{d}\,:\,\left|  \left|  a_{1}\right|
\right|  \leq N^{\beta} \\a_{2}\in\mathbb{Z}^{d}\,:\,\left|  \left|
a_{2}\right|  \right|  \leq N^{\beta} \\a_{3}\in\mathbb{Z}^{d}\,:\,\left|
\left|  a_{3}\right|  \right|  \leq N^{\beta}}}g_{p}^{\eta,K}\left(
0;a_{1},a_{2},a_{3}\right)  e^{\sum_{i=1,2,3}\left(  \nabla\xi_{p}\left(
[Nx_{i}]-x_{0}([N\mathbf{x}])\right)  ,a_{i}\right)  }\\
&  =\sum_{a_{1},a_{2},a_{3}\in\mathbb{Z}^{d}}g_{p}^{\eta,K}\left(
0;a_{1},a_{2},a_{3}\right)  e^{\sum_{i=1,2,3}\left(  t_{i},a_{i}\right)
}\left(  1+o\left(  1\right)  \right) \nonumber
\end{align}
and
\begin{equation}
\Theta_{p}\left(  \mathbf{x}\right)  :=\prod_{i=1,2,3}\frac{\tilde{\Lambda
}_{p}\left(  \frac{x_{i}-x_{0}\left(  \mathbf{x}\right)  }{\left|  \left|
x_{i}-x_{0}\left(  \mathbf{x}\right)  \right|  \right|  },t_{i}\left(
\mathbf{x}\right)  \right)  }{\sqrt{\left|  \left|  x_{i}-x_{0}\left(
\mathbf{x}\right)  \right|  \right|  ^{d-1}}}\sum_{a_{1},a_{2},a_{3}%
\in\mathbb{Z}^{d}}g_{p}^{\eta,K}\left(  0;a_{1},a_{2},a_{3}\right)
e^{\sum_{i=1,2,3}\left(  t_{i}\left(  \mathbf{x}\right)  ,a_{i}\right)  }%
\end{equation}
is an analytic function on $X_{3}^{\prime}.$

Furthermore,
\begin{align}
\varphi_{p,[N\mathbf{x}]}\left(  k\right)   &  =\varphi_{p,[N\mathbf{x}%
]}\left(  x_{0}\left(  [N\mathbf{x}]\right)  \right)  +\\
&  +\frac{1}{2N}\left(  k-x_{0}\left(  [N\mathbf{x}]\right)  ,H_{\varphi
}\left(  \frac{x_{0}\left(  [N\mathbf{x}]\right)  }{N},\frac{[N\mathbf{x}]}%
{N};p\right)  \left(  k-x_{0}\left(  [N\mathbf{x}]\right)  \right)  \right)
+O\left(  \frac{1}{\sqrt{N}}\right)  .\nonumber
\end{align}
Hence, by (\ref{stimHfi}), setting $k=[Nx_{0}\left(  \mathbf{x}\right)
+\sqrt{N}y],$ we obtain
\begin{equation}
\varphi_{p,[N\mathbf{x}]}\left(  [Nx_{0}\left(  \mathbf{x}\right)  +\sqrt
{N}y]\right)  =\varphi_{p,[N\mathbf{x}]}\left(  x_{0}\left(  [N\mathbf{x}%
]\right)  \right)  +\frac{1}{2}\left(  y,H_{\varphi}\left(  x_{0}\left(
\mathbf{x}\right)  ,\mathbf{x};p\right)  y\right)  +O\left(  \frac{1}{\sqrt
{N}}\right)  .
\end{equation}
\end{proof}

By (\ref{defA}), for any\ $\alpha\in\left(  \frac{1}{2},1\right)
,\,\mathbf{x\in}\left(  \mathbb{R}^{d}\right)  ^{3},$ we have
\begin{align}
E\left(  [N\mathbf{x}]\right)   &  =A_{\alpha,N}\left(  \mathbf{x}\right)
\bigvee A_{\alpha,N}^{\ast}\left(  \mathbf{x}\right)  ,\label{decE}\\
A_{\alpha,N}^{\ast}\left(  \mathbf{x}\right)   &  :=\bigcup_{k\in
\mathbb{Z}^{d}\,:\,\left|  \left|  k-x_{0}\left(  [N\mathbf{x}]\right)
\right|  \right|  \leq N^{\alpha}}F\left(  k;[N\mathbf{x}]\right)  ,
\end{align}
and by (\ref{decF}),
\begin{equation}
A_{\alpha,N}^{\ast}\left(  \mathbf{x}\right)  =\bigcup_{k\in\mathbb{Z}%
^{d}\,:\,\left|  \left|  k-x_{0}\left(  [N\mathbf{x}]\right)  \right|
\right|  \leq N^{\alpha}}\left\{  \bigvee_{\mathbf{b}\in\left(  \mathbb{Z}%
^{d}\right)  ^{3}}T\left(  \mathbf{b};k,[N\mathbf{x}]\right)  \bigvee T^{\ast
}\left(  k;[N\mathbf{x}]\right)  \right\}  .
\end{equation}
Thus, because of Proposition \ref{p1}, we are left with the estimate of the
events
\begin{align}
A_{\alpha,N}^{\ast\ast}\left(  \mathbf{x}\right)   &  :=\bigcup_{k\in
\mathbb{Z}^{d}\,:\,\left|  \left|  k-x_{0}\left(  [N\mathbf{x}]\right)
\right|  \right|  \leq N^{\alpha}}\bigvee_{\mathbf{b}\in\left(  \mathbb{Z}%
^{d}\right)  ^{3}}T\left(  \mathbf{b};k,[N\mathbf{x}]\right)  ,\\
T_{\alpha,N}^{\ast}\left(  \mathbf{x}\right)   &  :=\bigcup_{k\in
\mathbb{Z}^{d}\,:\,\left|  \left|  k-x_{0}\left(  [N\mathbf{x}]\right)
\right|  \right|  \leq N^{\alpha}}T^{\ast}\left(  k;[N\mathbf{x}]\right)
\label{defT*aN}%
\end{align}
but, as we have already remarked, given $k_{1},k_{2}\in\mathbb{Z}^{d}$ and a
$\mathbf{b}\in\left(  \mathbb{Z}^{d}\right)  ^{3}$ admissible for both $k_{1}$
and $k_{2},$ in general $T\left(  \mathbf{b};k_{1},[N\mathbf{x}]\right)  \cap
T\left(  \mathbf{b};k_{2},[N\mathbf{x}]\right)  \neq\emptyset.$ Hence we
cannot use simply the asymptotic estimate (\ref{lbPF}) and sum directly over $k.$

Let us define, for any $\mathbf{b\in}\left(  \mathbb{Z}^{d}\right)  ^{3},$ the
event $T\left(  \mathbf{b};\mathbf{[}N\mathbf{x]}\right)  :=\bigcup
_{k\in\mathbb{Z}^{d}}T\left(  \mathbf{b};k,[N\mathbf{x}]\right)  $ and notice
that, if $\mathbf{b}_{1}\neq\mathbf{b}_{2},$ then $T\left(  \mathbf{b}%
_{1};\mathbf{[}N\mathbf{x]}\right)  $ and $T\left(  \mathbf{b}_{2}%
;\mathbf{[}N\mathbf{x]}\right)  $ are disjoint.

We also remark that the probability of $T\left(  \mathbf{b};\mathbf{[}%
N\mathbf{x]}\right)  $ depends only on the vectors\linebreak \ $b_{i}%
-b_{j},\,i\neq j=1,2,3$ and therefore is translation invariant.

We are now ready to complete the proof of Theorem \ref{main}.

\begin{proof}
[Proof of Theorem \ref{main}]By the definition of $T\left(  \mathbf{b}%
;\mathbf{[}N\mathbf{x]}\right)  $ we have
\begin{equation}
E\left(  [N\mathbf{x}]\right)  =\bigvee_{b_{1},b_{2},b_{3}\in\mathbb{Z}^{d}%
}T\left(  b_{1},b_{2},b_{3};\mathbf{[}N\mathbf{x]}\right)  \bigvee T^{\ast
}\left(  \mathbf{[}N\mathbf{x]}\right)  ,
\end{equation}
where $T^{\ast}\left(  [N\mathbf{x}]\right)  :=\bigcup_{k\in\mathbb{Z}^{d}%
}T^{\ast}\left(  k;[N\mathbf{x}]\right)  .$ Hence
\begin{equation}
\mathbb{P}_{p}[E\left(  \left[  N\mathbf{x}\right]  \right)  ]=\sum
_{b_{1},b_{2},b_{3}\in\mathbb{Z}^{d}}\mathbb{P}_{p}[T\left(  b_{1},b_{2}%
,b_{3};[N\mathbf{x]}\right)  ]+\mathbb{P}_{p}[T^{\ast}\left(  [N\mathbf{x}%
]\right)  ].
\end{equation}
Proceeding as in the proof of the previous proposition, for any $\mathbf{b}%
\in\left(  \mathbb{Z}^{d}\right)  ^{3},$ we define the function
\begin{equation}
g_{p}^{\eta,K}\left(  \mathbf{b}\right)  =g_{p}^{\eta,K}\left(  b_{1}%
,b_{2},b_{3}\right)  ,
\end{equation}
which is the probability of the event $G^{\eta,K}\left(  \mathbf{b}\right)
:=\bigcup_{k\in\mathcal{T}\left(  \mathbf{b}\right)  }G^{\eta,K}\left(
k;\mathbf{b}\right)  .$ Then
\begin{equation}
\mathbb{P}_{p}[T\left(  \mathbf{b,[}N\mathbf{x]}\right)  ]:=g_{p}^{\eta
,K}\left(  \mathbf{b}\right)  \prod_{i=1,2,3}\tilde{h}_{t_{i}}^{\eta,K}\left(
[Nx_{i}]-b_{i}\right)  .
\end{equation}
By the translation invariance of $g_{p}^{\eta,K}\left(  \mathbf{b}\right)  ,$
setting $b_{1}=b,\,b_{2}=b+a_{1},\,b_{3}=b+a_{2},$ we obtain
\begin{align}
&  \sum_{b_{1},b_{2},b_{3}\in\mathbb{Z}^{d}}\mathbb{P}_{p}[T\left(
b_{1},b_{2},b_{3};[N\mathbf{x]}\right)  ]=\sum_{b_{1},b_{2},b_{3}\in
\mathbb{Z}^{d}}g_{p}^{\eta,K}\left(  b_{1},b_{2},b_{3}\right)  \prod
_{i=1,2,3}\tilde{h}_{t_{i}}^{\eta,K}\left(  [Nx_{i}]-b_{i}\right) \\
&  =e^{-\varphi_{p,[N\mathbf{x}]}\left(  x_{0}\left(  [N\mathbf{x}]\right)
\right)  }\sum_{b\in\mathbb{Z}^{d}}\tilde{h}_{t_{1}}^{\eta,K}\left(
[Nx_{1}]-b\right)  e^{\xi_{p}\left(  [Nx_{1}]-x_{0}\left(  [N\mathbf{x}%
]\right)  \right)  }\times\nonumber\\
&  \times\sum_{a_{1},a_{2}\in\mathbb{Z}^{d}}g_{p}^{\eta,K}\left(
0,a_{1},a_{2}\right)  \prod_{i=2,3}e^{\xi_{p}\left(  [Nx_{i}]-x_{0}\left(
[N\mathbf{x}]\right)  \right)  }\tilde{h}_{t_{i}}^{\eta,K}\left(
[Nx_{i}]-b-a_{i-1}\right)  .\nonumber
\end{align}
Since,
\begin{equation}
\mathcal{T}\left(  \mathbf{b}\right)  =\mathcal{T}\left(  b,b+a_{1}%
,b+a_{2}\right)  =b+\mathcal{T}\left(  0\mathbf{,}a_{1},a_{3}\right)  ,
\end{equation}
by (\ref{stimghK}),
\begin{gather}
e^{\varphi_{p,[N\mathbf{x}]}\left(  x_{0}\left(  [N\mathbf{x}]\right)
\right)  }\sum_{b\in\mathbb{Z}^{d}}\tilde{h}_{t_{1}}^{\eta,K}\left(
[Nx_{1}]-b\right)  \sum_{a_{1},a_{2}\in\mathbb{Z}^{d}}g_{p}^{\eta,K}\left(
0,a_{1},a_{2}\right)  \prod_{i=2,3}\tilde{h}_{t_{i}}^{\eta,K}\left(
[Nx_{i}]-b-a_{i-1}\right) \\
\leq\sum_{b\in\mathbb{Z}^{d}}\sum_{a_{1},a_{2}\in\mathbb{Z}^{d}}\sum
_{k\in\mathcal{T}\left(  0,a_{1},a_{2};[N\mathbf{x}]\right)  \cap
\mathbb{Z}^{d}}\exp\left\{  -c_{12}\left[  \left|  \left|  k\right|  \right|
+\left|  \left|  k-a_{1}\right|  \right|  +\left|  \left|  k-a_{2}\right|
\right|  \right]  +\right. \nonumber\\
-\xi_{p}\left(  k\right)  -\xi_{p}\left(  k-a_{1}\right)  -\xi_{p}\left(
k-a_{2}\right)  -\xi_{p}\left(  [Nx_{1}]-b\right)  +\xi_{p}\left(
[Nx_{1}]-x_{0}\left(  [N\mathbf{x}]\right)  \right) \nonumber\\
\left.  -\xi_{p}\left(  [Nx_{2}]-b-a_{1}\right)  +\xi_{p}\left(
[Nx_{2}]-x_{0}\left(  [N\mathbf{x}]\right)  \right)  -\xi_{p}\left(
[Nx_{3}]-b-a_{2}\right)  +\xi_{p}\left(  [Nx_{3}]-x_{0}\left(  [N\mathbf{x}%
]\right)  \right)  \right\}  .\nonumber
\end{gather}

We recall that the probability that points in $\mathcal{T}\left(
\mathbf{b}\right)  $\ disjointly connected to $[Nx_{1}],[Nx_{2}],[Nx_{3}],$
lie outside of a neighborhood of $x_{0}\left(  [N\mathbf{x}]\right)  $ of
radius $N^{\alpha},$ with $\alpha>\frac{1}{2},$ is smaller than the r.h.s. of
(\ref{3BK}). Hence, we can restrict ourselves to consider only those
configurations of points $b_{1},b_{2},b_{3},$ such that the associated set
$\mathcal{T}\left(  \mathbf{b}\right)  $\ has non-empty intersection with
$N^{\alpha}\mathbf{U}^{p}\left(  x_{0}\left(  [N\mathbf{x}]\right)  \right)
.$ Making use of the shorthand notation $\overset{\prime}{\sum_{k}}$ for
$\sum_{k\in\mathcal{T}\left(  0\mathbf{,}a_{1},a_{3}\right)  \cap N^{\alpha
}\mathbf{U}^{p}\left(  x_{0}\left(  [N\mathbf{x}]\right)  -b\right)
\cap\mathbb{Z}^{d}},$ by the convexity of $\xi_{p}$ and Lemma \ref{lbqf}, we
obtain
\begin{gather}
\sum_{b\in\mathbb{Z}^{d}}\sum_{a_{1},a_{2}\in\mathbb{Z}^{d}}\sum_{k}^{\prime
}\exp\left\{  -c_{12}\left[  \left|  \left|  k\right|  \right|  +\left|
\left|  k-a_{1}\right|  \right|  +\left|  \left|  k-a_{2}\right|  \right|
\right]  +\right. \label{stimthm}\\
-\xi_{p}\left(  k\right)  -\xi_{p}\left(  k-a_{1}\right)  -\xi_{p}\left(
k-a_{2}\right)  -\xi_{p}\left(  [Nx_{1}]-b\right)  +\xi_{p}\left(
[Nx_{1}]-x_{0}\left(  [N\mathbf{x}]\right)  \right) \nonumber\\
\left.  -\xi_{p}\left(  [Nx_{2}]-b-a_{1}\right)  +\xi_{p}\left(
[Nx_{2}]-x_{0}\left(  [N\mathbf{x}]\right)  \right)  -\xi_{p}\left(
[Nx_{3}]-b-a_{2}\right)  +\xi_{p}\left(  [Nx_{3}]-x_{0}\left(  [N\mathbf{x}%
]\right)  \right)  \right\} \nonumber\\
\leq\sum_{b\in\mathbb{Z}^{d}}\sum_{a_{1},a_{2}\in\mathbb{Z}^{d}}\sum
_{k}^{\prime}\exp\left\{  -c_{12}\left[  \left|  \left|  k\right|  \right|
+\left|  \left|  k-a_{1}\right|  \right|  +\left|  \left|  k-a_{2}\right|
\right|  \right]  +\right. \nonumber\\
\left.  -[\varphi_{p,[N\mathbf{x}]}\left(  b+k\right)  -\varphi
_{p,[N\mathbf{x}]}\left(  x_{0}\left(  [N\mathbf{x}]\right)  \right)
]\right\} \nonumber\\
\leq\sum_{b\in\mathbb{Z}^{d}}\sum_{a_{1},a_{2}\in\mathbb{Z}^{d}}\sum
_{k}^{\prime}\exp\left\{  -c_{12}\left[  \left|  \left|  k\right|  \right|
+\left|  \left|  k-a_{1}\right|  \right|  +\left|  \left|  k-a_{2}\right|
\right|  \right]  -\frac{c_{2}}{N}\left|  \left|  x_{0}\left(  [N\mathbf{x}%
]\right)  -\left(  b+k\right)  \right|  \right|  ^{2}\right\}  .\nonumber
\end{gather}
Thus, for $\left|  \left|  x_{0}\left(  [N\mathbf{x}]\right)  -b\right|
\right|  >N^{\alpha},$ denoting by $y=y\left(  b,a_{1},a_{2}\right)
\in\mathbb{R}^{d} $ the minimizing point of the convex function
\begin{equation}
w\left(  z\right)  :=\frac{c_{2}}{N}\left|  \left|  x_{0}\left(
[N\mathbf{x}]\right)  -\left(  b+z\right)  \right|  \right|  ^{2}%
+c_{12}\left[  \left|  \left|  z\right|  \right|  +\left|  \left|
z-a_{1}\right|  \right|  +\left|  \left|  z-a_{2}\right|  \right|  \right]  ,
\end{equation}
if $\left|  \left|  y\right|  \right|  \geq\frac{N^{\alpha}}{2},$ then
(\ref{stimthm}) is smaller than $e^{-c_{16}\frac{N^{\alpha}}{2}}.$ On the
other hand, if $\left|  \left|  y\right|  \right|  <\frac{N^{\alpha}}{2},$
then (\ref{stimthm}) is smaller than $e^{-c_{17}\frac{N^{2\alpha-1}}{4}}.$
Therefore, setting $\alpha^{\prime}:=\alpha\wedge\left(  2\alpha-1\right)  ,$
for sufficiently large value of $N,$ we get
\begin{align}
&  e^{\varphi_{p,[N\mathbf{x}]}\left(  x_{0}\left(  [N\mathbf{x}]\right)
\right)  }\sum_{b\in\mathbb{Z}^{d}\,:\,\left|  \left|  b-x_{0}\left(  \left[
N\mathbf{x}\right]  \right)  \right|  \right|  >N^{\alpha}}\tilde{h}_{t_{1}%
}^{\eta,K}\left(  [Nx_{1}]-b\right)  \sum_{a_{1},a_{2}\in\mathbb{Z}^{d}}%
g_{p}^{\eta,K}\left(  0,a_{1},a_{2}\right)  \times\label{pb-x0>Na}\\
&  \times\prod_{i=2,3}\tilde{h}_{t_{i}}^{\eta,K}\left(  [Nx_{i}]-b-a_{i-1}%
\right)  \leq e^{-c_{18}N^{\alpha^{\prime}}}.\nonumber
\end{align}
Moreover, for any $\beta\in\left(  0,\frac{1}{2}\right)  ,$
\begin{gather}
e^{\varphi_{p,[N\mathbf{x}]}\left(  x_{0}\left(  [N\mathbf{x}]\right)
\right)  }\sum_{b\in\mathbb{Z}^{d}\,:\,\left|  \left|  b-x_{0}\left(  \left[
N\mathbf{x}\right]  \right)  \right|  \right|  \leq N^{\alpha}}\tilde
{h}_{t_{1}}^{\eta,K}\left(  [Nx_{1}]-b\right)  \times\label{pa12>Nb}\\
\times\sum_{a_{1}\in\mathbb{Z}^{d}\,:\,\left|  \left|  a_{1}\right|  \right|
>N^{\beta}}\sum_{a_{2}\in\mathbb{Z}^{d}}g_{p}^{\eta,K}\left(  0,a_{1}%
,a_{2}\right)  \prod_{i=1,2}\tilde{h}_{t_{i+1}}^{\eta,K}\left(  [Nx_{i+1}%
]-b-a_{i}\right) \nonumber\\
\leq\sum_{b\in\mathbb{Z}^{d}\,:\,\left|  \left|  b-x_{0}\left(  \left[
N\mathbf{x}\right]  \right)  \right|  \right|  \leq N^{\alpha}}\sum_{a_{1}%
\in\mathbb{Z}^{d}\,:\,\left|  \left|  a_{1}\right|  \right|  >N^{\beta}}%
\sum_{a_{2}\in\mathbb{Z}^{d}}\times\nonumber\\
\times\sum_{k}^{\prime}e^{-\frac{c_{2}}{N}\left|  \left|  x_{0}\left(
[N\mathbf{x}]\right)  -\left(  b+k\right)  \right|  \right|  ^{2}%
-c_{12}\left[  \left|  \left|  k\right|  \right|  +\left|  \left|
k-a_{1}\right|  \right|  +\left|  \left|  k-a_{2}\right|  \right|  \right]
}\nonumber\\
\leq c_{19}N^{2\alpha d}\sum_{a_{1}\in\mathbb{Z}^{d}\,:\,\left|  \left|
a_{1}\right|  \right|  >N^{\beta}}e^{-c_{12}\left|  \left|  a_{1}\right|
\right|  }\leq e^{-c_{20}N^{\beta}}.\nonumber
\end{gather}
A similar inequality holds with $a_{1}$ and $a_{2}$ exchanged.

Notice that (\ref{stim1}), (\ref{pa12>Nb}) and (\ref{pb-x0>Na}) imply
\begin{equation}
\mathbb{P}_{p}\left[  F\left(  k_{1};[N\mathbf{x}]\right)  \cap F\left(
k_{2};[N\mathbf{x}]\right)  \cap\{\left|  \left|  k_{1}-k_{2}\right|  \right|
>N^{\beta}\}\right]  \leq e^{-\varphi_{p,[N\mathbf{x}]}\left(  x_{0}\left(
[N\mathbf{x}]\right)  \right)  -c_{21}N^{\beta\wedge\alpha^{\prime}}}%
\end{equation}
and so what stated in Remark \ref{rem1}.

Then we are left with the estimate of
\begin{align}
&  \sum_{b\in\mathbb{Z}^{d}\,:\,\left|  \left|  b-x_{0}\left(  \left[
N\mathbf{x}\right]  \right)  \right|  \right|  \leq N^{\alpha}}e^{\xi
_{p}\left(  [Nx_{1}]-x_{0}\left(  \left[  N\mathbf{x}\right]  \right)
\right)  }\tilde{h}_{t_{1}}^{\eta,K}\left(  [Nx_{1}]-b\right)  \times
\label{stimfin}\\
&  \times\sum_{\substack{a_{1}\in\mathbb{Z}^{d}\,:\,\left|  \left|
a_{1}\right|  \right|  \leq N^{\beta} \\a_{2}\in\mathbb{Z}^{d}\,:\,\left|
\left|  a_{2}\right|  \right|  \leq N^{\beta}}}g_{p}^{\eta,K}\left(
0,a_{1},a_{2}\right)  \prod_{i=1,2}e^{\xi_{p}\left(  [Nx_{i+1}]-x_{0}\left(
\left[  N\mathbf{x}\right]  \right)  \right)  }\tilde{h}_{t_{i+1}}^{\eta
,K}\left(  [Nx_{i+1}]-b-a_{i}\right)  .\nonumber
\end{align}
Now we choose $\alpha=\frac{1}{2}+\varepsilon$ and $\beta<\frac{1}%
{2}-\varepsilon$ with $\varepsilon\in\left(  0,\frac{1}{6}\right)  $. For
$\left|  \left|  b-x_{0}\left(  \left[  N\mathbf{x}\right]  \right)  \right|
\right|  \leq N^{\alpha}$ and $\left|  \left|  a_{1}\right|  \right|  ,\left|
\left|  a_{2}\right|  \right|  \leq N^{\beta}$ we get
\begin{gather}
\xi_{p}\left(  [Nx_{1}]-b\right)  -\xi_{p}\left(  [Nx_{1}]-x_{0}\left(
\left[  N\mathbf{x}\right]  \right)  \right)  =-\left(  \nabla\xi_{p}\left(
[Nx_{1}]-x_{0}\left(  \left[  N\mathbf{x}\right]  \right)  \right)
,b-x_{0}\left(  \left[  N\mathbf{x}\right]  \right)  \right)  +\label{stimxib}%
\\
+\frac{1}{2}\left(  b-x_{0}\left(  \left[  N\mathbf{x}\right]  \right)
,H_{\xi}\left(  [Nx_{1}]-x_{0}\left(  \left[  N\mathbf{x}\right]  \right)
;p\right)  \left(  b-x_{0}\left(  \left[  N\mathbf{x}\right]  \right)
\right)  \right)  +O\left(  N^{-\frac{1}{2}+3\varepsilon}\right) \nonumber\\
\xi\left(  \lbrack Nx_{i+1}]-b-a_{i}\right)  -\xi_{p}\left(  [Nx_{i+1}%
]-x_{0}\left(  \left[  N\mathbf{x}\right]  \right)  \right)  =-\left(
\nabla\xi_{p}\left(  [Nx_{i+1}]-x_{0}\left(  \left[  N\mathbf{x}\right]
\right)  \right)  ,b+a_{i}-x_{0}\left(  \left[  N\mathbf{x}\right]  \right)
\right)  +\nonumber\\
+\frac{1}{2}\left(  b+a_{i}-x_{0}\left(  \left[  N\mathbf{x}\right]  \right)
,H_{\xi}\left(  [Nx_{i+1}]-x_{0}\left(  \left[  N\mathbf{x}\right]  \right)
;p\right)  \left(  b+a_{i}-x_{0}\left(  \left[  N\mathbf{x}\right]  \right)
\right)  \right)  +O\left(  N^{-\frac{1}{2}+3\varepsilon}\right) \nonumber\\
=-\left(  \nabla\xi_{p}\left(  [Nx_{i+1}]-x_{0}\left(  \left[  N\mathbf{x}%
\right]  \right)  \right)  ,b+a_{i}-x_{0}\left(  \left[  N\mathbf{x}\right]
\right)  \right)  +\nonumber\\
+\frac{1}{2}\left(  b-x_{0}\left(  \left[  N\mathbf{x}\right]  \right)
,H_{\xi}\left(  [Nx_{i+1}]-x_{0}\left(  \left[  N\mathbf{x}\right]  \right)
;p\right)  \left(  b-x_{0}\left(  \left[  N\mathbf{x}\right]  \right)
\right)  \right)  +\nonumber\\
+O\left(  N^{-\frac{1}{2}+\varepsilon+\beta\vee(2\varepsilon)}\right)  \quad
i=1,2.\nonumber
\end{gather}
Moreover, since $\sum_{i=1,2,3}\nabla\xi_{p}\left(  [Nx_{i}]-x_{0}\left(
\left[  N\mathbf{x}\right]  \right)  \right)  =0,$ for any $x\in\mathbb{R}%
^{d},$
\begin{equation}
\sum_{i=1,2,3}\left(  \nabla\xi_{p}\left(  [Nx_{i}]-x_{0}\left(  \left[
N\mathbf{x}\right]  \right)  \right)  ,x\right)  =0.\label{mincond}%
\end{equation}
Then,
\begin{align}
&  \xi_{p}\left(  [Nx_{1}]-b\right)  +\sum_{i=1,2}\xi\left(  \lbrack
Nx_{i+1}]-b-a_{i}\right)  -\varphi_{p,[N\mathbf{x}]}\left(  x_{0}\left(
\left[  N\mathbf{x}\right]  \right)  \right) \\
&  =-\sum_{i=1,2}\left(  \nabla\xi_{p}\left(  [Nx_{i+1}]-x_{0}\left(  \left[
N\mathbf{x}\right]  \right)  \right)  ,a_{i}\right)  +\nonumber\\
&  +\frac{1}{2}\left(  b-x_{0}\left(  \left[  N\mathbf{x}\right]  \right)
,H_{\varphi}\left(  x_{0}\left(  \left[  N\mathbf{x}\right]  \right)  ,\left[
N\mathbf{x}\right]  ;p\right)  \left(  b-x_{0}\left(  \left[  N\mathbf{x}%
\right]  \right)  \right)  \right)  +O\left(  N^{-\frac{1}{2}+\beta\vee\left(
2\varepsilon\right)  +\varepsilon}\right)  .\nonumber
\end{align}
Hence, making use of (\ref{ashtt}), (\ref{stimfin}) becomes
\begin{align}
&  \prod_{i=1,2,3}\frac{\tilde{\Lambda}_{p}\left(  \frac{x_{i}-x_{0}\left(
\mathbf{x}\right)  }{\left|  \left|  x_{i}-x_{0}\left(  \mathbf{x}\right)
\right|  \right|  },t_{i}\right)  }{\sqrt{2\pi N^{d-1}\left|  \left|
x_{i}-x_{0}\left(  \mathbf{x}\right)  \right|  \right|  ^{d-1}}}\sum
_{b\in\mathbb{Z}^{d}\,:\,\left|  \left|  b-x_{0}\left(  \left[  N\mathbf{x}%
\right]  \right)  \right|  \right|  \leq N^{\frac{1}{2}+\varepsilon}}%
e^{-\frac{1}{2}\left(  b-x_{0}\left(  \left[  N\mathbf{x}\right]  \right)
,H_{\varphi}\left(  x_{0}\left(  \left[  N\mathbf{x}\right]  \right)  ,\left[
N\mathbf{x}\right]  ;p\right)  \left(  b-x_{0}\left(  \left[  N\mathbf{x}%
\right]  \right)  \right)  \right)  }\times\label{PE1}\\
&  \times\sum_{\substack{a_{1}\in\mathbb{Z}^{d}\,:\,\left|  \left|
a_{1}\right|  \right|  \leq N^{\beta} \\a_{2}\in\mathbb{Z}^{d}\,:\,\left|
\left|  a_{2}\right|  \right|  \leq N^{\beta}}}g_{p}^{\eta,K}\left(
0,a_{1},a_{2}\right)  e^{\sum_{i=1,2}\left(  \nabla\xi_{p}\left(
[Nx_{i+1}]-x_{0}\left(  \left[  N\mathbf{x}\right]  \right)  \right)
,a_{i}\right)  }\left(  1+o\left(  1\right)  \right)  .\nonumber
\end{align}
Furthermore, by (\ref{stimghK}) and (\ref{mincond}),
\begin{align}
&  \sum_{\substack{a_{1}\in\mathbb{Z}^{d}\,:\,\left|  \left|  a_{1}\right|
\right|  \leq N^{\beta} \\a_{2}\in\mathbb{Z}^{d}\,:\,\left|  \left|
a_{2}\right|  \right|  \leq N^{\beta}}}g_{p}^{\eta,K}\left(  0,a_{1}%
,a_{2}\right)  e^{\sum_{i=1,2}\left(  \nabla\xi_{p}\left(  [Nx_{i+1}%
]-x_{0}\left(  \left[  N\mathbf{x}\right]  \right)  \right)  ,a_{i}\right)
}\label{stimfin1}\\
&  \leq\sum_{\substack{a_{1}\in\mathbb{Z}^{d}\,:\,\left|  \left|
a_{1}\right|  \right|  \leq N^{\beta} \\a_{2}\in\mathbb{Z}^{d}\,:\,\left|
\left|  a_{2}\right|  \right|  \leq N^{\beta}}}\sum_{k}^{\prime}e^{-\xi
_{p}\left(  k\right)  -\sum_{i=1,2}[\xi_{p}\left(  k-a_{i}\right)  -\left(
\nabla\xi_{p}\left(  [Nx_{i+1}]-x_{0}\left(  \left[  N\mathbf{x}\right]
\right)  \right)  ,a_{i}\right)  ]-c_{12}\left[  \left|  \left|  k\right|
\right|  +\sum_{i=1,2}\left|  \left|  k-a_{i}\right|  \right|  \right]
}\nonumber\\
&  \leq\sum_{\substack{a_{1}\in\mathbb{Z}^{d}\,:\,\left|  \left|
a_{1}\right|  \right|  \leq N^{\beta} \\a_{2}\in\mathbb{Z}^{d}\,:\,\left|
\left|  a_{2}\right|  \right|  \leq N^{\beta}}}\sum_{k}^{\prime}e^{-\xi
_{p}\left(  k\right)  -\left(  \nabla\xi_{p}\left(  [Nx_{1}]-x_{0}\left(
\left[  N\mathbf{x}\right]  \right)  \right)  ,k\right)  -\sum_{i=1,2}[\xi
_{p}\left(  k-a_{i}\right)  -\left(  \nabla\xi_{p}\left(  [Nx_{i+1}%
]-x_{0}\left(  \left[  N\mathbf{x}\right]  \right)  \right)  ,a_{i}-k\right)
]}\times\nonumber\\
&  \times e^{-c_{12}\left[  \left|  \left|  k\right|  \right|  +\sum
_{i=1,2}\left|  \left|  k-a_{i}\right|  \right|  \right]  }.\nonumber
\end{align}
But,
\begin{gather}
\xi_{p}\left(  [Nx_{i}]-\left(  b+k\right)  \right)  =\xi_{p}\left(
[Nx_{i}]-x_{0}\left(  \left[  N\mathbf{x}\right]  \right)  \right)  -\left(
\nabla\xi_{p}\left(  [Nx_{i}]-x_{0}\left(  \left[  N\mathbf{x}\right]
\right)  \right)  ,b+k-x_{0}\left(  \left[  N\mathbf{x}\right]  \right)
\right)  +\\
+\frac{1}{2}\left(  b+k-x_{0}\left(  \left[  N\mathbf{x}\right]  \right)
,H_{\xi}\left(  [Nx_{i}]-x_{0}\left(  \left[  N\mathbf{x}\right]  \right)
;p\right)  \left(  b+k-x_{0}\left(  \left[  N\mathbf{x}\right]  \right)
\right)  \right)  +O\left(  N^{-\frac{1}{2}+3\varepsilon}\right) \nonumber
\end{gather}
and by (\ref{stimxib}) it follows that
\begin{align}
\xi_{p}\left(  k\right)   &  \geq\xi_{p}\left(  [Nx_{i}]-\left(  b+k\right)
\right)  -\xi_{p}\left(  [Nx_{i}]-b\right) \\
&  =-\left(  \nabla\xi_{p}\left(  [Nx_{i}]-x_{0}\left(  \left[  N\mathbf{x}%
\right]  \right)  \right)  ,k\right)  +O\left(  N^{-\frac{1}{2}+\beta
+\varepsilon}\right)  .\nonumber
\end{align}
Then, since for any $x,y\in\mathbb{R}^{d},\,\left(  \nabla\xi_{p}\left(
x\right)  ,y\right)  \leq\xi_{p}\left(  y\right)  ,$ from (\ref{stimfin1}) it
follows that there exists a positive constant $c_{22}$ such that
\begin{align}
&  \sum_{\substack{a_{1}\in\mathbb{Z}^{d}\,:\,\left|  \left|  a_{1}\right|
\right|  \leq N^{\beta} \\a_{2}\in\mathbb{Z}^{d}\,:\,\left|  \left|
a_{2}\right|  \right|  \leq N^{\beta}}}g_{p}^{\eta,K}\left(  0,a_{1}%
,a_{2}\right)  e^{\sum_{i=1,2}\left(  \nabla\xi_{p}\left(  [Nx_{i+1}%
]-x_{0}\left(  \left[  N\mathbf{x}\right]  \right)  \right)  ,a_{i}\right)
}\\
&  \leq\sum_{k\in\mathbb{Z}^{d}}\sum_{\substack{a_{1}\in\mathbb{Z}%
^{d}\,:\,\left|  \left|  a_{1}\right|  \right|  \leq N^{\beta} \\a_{2}%
\in\mathbb{Z}^{d}\,:\,\left|  \left|  a_{2}\right|  \right|  \leq N^{\beta}%
}}e^{-c_{22}\left[  \left|  \left|  k\right|  \right|  +\sum_{i=1,2}\left|
\left|  k-a_{i}\right|  \right|  \right]  }.\nonumber
\end{align}
Hence, the r.h.s. of (\ref{stimfin1}) is bounded by a finite constant and
\begin{align}
&  \sum_{\substack{a_{1}\in\mathbb{Z}^{d}\,:\,\left|  \left|  a_{1}\right|
\right|  \leq N^{\beta} \\a_{2}\in\mathbb{Z}^{d}\,:\,\left|  \left|
a_{2}\right|  \right|  \leq N^{\beta}}}g_{p}^{\eta,K}\left(  0,a_{1}%
,a_{2}\right)  e^{\sum_{i=1,2}\left(  \nabla\xi_{p}\left(  [Nx_{i+1}%
]-x_{0}\left(  \left[  N\mathbf{x}\right]  \right)  \right)  ,a_{i}\right)
}\label{stimgfin}\\
&  =\sum_{a_{1},a_{2}\in\mathbb{Z}^{d}}g_{p}^{\eta,K}\left(  0,a_{1}%
,a_{2}\right)  e^{\sum_{i=1,2}\left(  t_{i+1},a_{i}\right)  }\left(
1+o\left(  1\right)  \right)  .\nonumber
\end{align}
Finally, by (\ref{stimPT*}),
\begin{equation}
\mathbb{P}_{p}[T^{\ast}\left(  [N\mathbf{x}]\right)  ]\leq\sum_{i=1,2,3}%
\sum_{k\in\mathbb{Z}^{d}}e^{-\varphi_{p,[N\mathbf{x]}}\left(  k\right)
-c_{10}\left|  \left|  k-[Nx_{i}]\right|  \right|  }.
\end{equation}
By (\ref{defT*aN}), we need only to estimate
\begin{equation}
\mathbb{P}_{p}[T_{\frac{1}{2}+\varepsilon,N}^{\ast}\left(  \mathbf{x}\right)
]\leq\sum_{i=1,2,3}\sum_{k\in\mathbb{Z}^{d}\,:\,\left|  \left|  k-x_{0}\left(
\left[  N\mathbf{x}\right]  \right)  \right|  \right|  \leq N^{\frac{1}%
{2}+\varepsilon}}e^{-\varphi_{p,[N\mathbf{x]}}\left(  k\right)  -c_{10}\left|
\left|  k-[Nx_{i}]\right|  \right|  },
\end{equation}
but, by Lemma \ref{lbqf},
\begin{equation}
\varphi_{p,[N\mathbf{x]}}\left(  k\right)  \geq\varphi_{p,[N\mathbf{x]}%
}\left(  x_{0}\left(  [N\mathbf{x}]\right)  \right)  +\frac{c_{2}}{N}\left|
\left|  k-x_{0}\left(  [N\mathbf{x}]\right)  \right|  \right|  ^{2}.
\end{equation}
Thus, there exists a positive constant $c_{23}$ such that,
\begin{equation}
\mathbb{P}_{p}[T^{\ast}\left(  [N\mathbf{x}]\right)  ]\leq N^{d\left(
\frac{1}{2}+\varepsilon\right)  }e^{-\varphi_{p,[N\mathbf{x]}}\left(
x_{0}\left(  [N\mathbf{x}]\right)  \right)  -c_{23}N}.
\end{equation}

Collecting all the previous estimates, from (\ref{PE1}), (\ref{stimHfi}) and
(\ref{stimgfin}) we obtain
\begin{equation}
\mathbb{P}_{p}[E\left(  [N\mathbf{x}]\right)  ]=\frac{\left(  2\pi\right)
^{\frac{d}{2}}N^{\frac{d}{2}}}{\sqrt{\det H_{\varphi}\left(  x_{0}\left(
\mathbf{x}\right)  ,\mathbf{x};p\right)  }}\frac{\theta_{p}\left(
\mathbf{x}\right)  }{\left(  2\pi N^{d-1}\right)  ^{\frac{3}{2}}}%
e^{-\varphi_{p,[N\mathbf{x]}}\left(  x_{0}\left(  [N\mathbf{x}]\right)
\right)  }\left(  1+o\left(  1\right)  \right)  ,
\end{equation}
with
\begin{equation}
\theta_{p}\left(  \mathbf{x}\right)  :=\prod_{i=1,2,3}\frac{\tilde{\Lambda
}_{p}\left(  \frac{x_{i}-x_{0}\left(  \mathbf{x}\right)  }{\left|  \left|
x_{i}-x_{0}\left(  \mathbf{x}\right)  \right|  \right|  },t_{i}\left(
\mathbf{x}\right)  \right)  }{\sqrt{\left|  \left|  x_{i}-x_{0}\left(
\mathbf{x}\right)  \right|  \right|  ^{d-1}}}\sum_{a_{1},a_{2}\in
\mathbb{Z}^{d}}g_{p}^{\eta,K}\left(  0,a_{1},a_{2}\right)  e^{\sum
_{i=1,2}\left(  t_{i+1}\left(  \mathbf{x}\right)  ,a_{i}\right)  }%
\end{equation}
analytic function on $X_{3}^{\prime}.$

Therefore, by (\ref{lbPF}),
\begin{align}
&  \mathbb{P}_{p}\left[  F\left(  \left[  Nx_{0}\left(  \mathbf{x}\right)
+y\sqrt{N}\right]  ;\left[  N\mathbf{x}\right]  \right)  |E\left(  \left[
N\mathbf{x}\right]  \right)  \right] \\
&  =\frac{\mathbb{P}_{p}\left[  F\left(  \left[  Nx_{0}\left(  \mathbf{x}%
\right)  +y\sqrt{N}\right]  ;\left[  N\mathbf{x}\right]  \right)  \right]
}{\mathbb{P}_{p}\left[  E\left(  \left[  N\mathbf{x}\right]  \right)  \right]
}\nonumber\\
&  =\frac{\Theta_{p}\left(  \mathbf{x}\right)  \sqrt{\det H_{\varphi}\left(
x_{0}\left(  \mathbf{x}\right)  ,\mathbf{x};p\right)  }}{\left(  2\pi\right)
^{\frac{d}{2}}\theta_{p}\left(  \mathbf{x}\right)  N^{\frac{d}{2}}}%
e^{-\frac{1}{2}\left(  y,H_{\varphi}\left(  x_{0}\left(  \mathbf{x}\right)
,\mathbf{x},p\right)  y\right)  }\left(  1+o\left(  1\right)  \right)
,\nonumber
\end{align}
which gives the asymptotic estimate (\ref{PFE}) with
\begin{equation}
\Phi_{p}\left(  \mathbf{x}\right)  :=\frac{\Theta_{p}\left(  \mathbf{x}%
\right)  }{\theta_{p}\left(  \mathbf{x}\right)  }=\frac{\sum_{a_{1}%
,a_{2},a_{3}\in\mathbb{Z}^{d}}g_{p}^{\eta,K}\left(  0;a_{1},a_{2}%
,a_{3}\right)  e^{\sum_{i=1,2,3}\left(  t_{i}\left(  \mathbf{x}\right)
,a_{i}\right)  }}{\sum_{a_{1},a_{2}\in\mathbb{Z}^{d}}g_{p}^{\eta,K}\left(
0,a_{1},a_{2}\right)  e^{\sum_{i=1,2}\left(  t_{i+1}\left(  \mathbf{x}\right)
,a_{i}\right)  }}.
\end{equation}
\end{proof}

\end{document}